\documentclass[10pt]{article}

\usepackage{amsfonts}
\usepackage{amsxtra}
\usepackage{amsthm}
\usepackage{amssymb}
\usepackage{amsmath,amscd}
\usepackage{supertabular}
\usepackage{pstricks}
\usepackage{pst-plot}
\usepackage{pstricks-add}
\usepackage{pst-eps}
\usepackage{makeidx}
\usepackage{bm}
\usepackage{stmaryrd}

\usepackage{tikz}

\newtheorem{theorem}{Theorem}[section]
\newtheorem{lemma}[theorem]{Lemma}
\newtheorem{definition}[theorem]{Definition}
\newtheorem{proposition}[theorem]{Proposition}

\newtheorem{corollary}[theorem]{Corollary}

\DeclareMathOperator\USp{USp}
\DeclareMathOperator\rk{rk}

\newcommand{\eps}{\epsilon}

\newcommand{\bbC}{\mathbb{C}}
\newcommand{\bbD}{\mathbb{D}}
\newcommand{\bbZ}{\mathbb{Z}}
\newcommand{\bbN}{\mathbb{N}}

\newcommand{\calA}{\mathcal{A}}

\newcommand{\calD}{\mathcal{D}}
\newcommand{\calE}{\mathcal{E}}

\newcommand{\calR}{\mathcal{R}}

\newcommand{\del}{\partial}

\newcommand{\tr}{\mathrm{tr}}
\newcommand{\End}{\mathrm{End}}

\newcommand{\Hom}{\mathrm{Hom}}

\newcommand{\diag}{\mathrm{diag}}

\newcommand{\laa}{\mathfrak{a}}

\newcommand{\lag}{\mathfrak{g}}
\newcommand{\lah}{\mathfrak{h}}

\newcommand{\lak}{\mathfrak{k}}
\newcommand{\lam}{\mathfrak{m}}
\newcommand{\lan}{\mathfrak{n}}

\newcommand{\laq}{\mathfrak{q}}

\newcommand{\lat}{\mathfrak{t}}

\newcommand{\laso}{\mathfrak{so}}

\newcommand{\laZ}{\mathfrak{Z}}

\newcommand{\SO}{\mathrm{SO}}

\newcommand{\Spin}{\mathrm{Spin}}
\newcommand{\SU}{\mathrm{SU}}
\newcommand{\U}{\mathrm{U}}

\newcommand{\Ad}{\mathrm{Ad}}

\newcommand{\sph}{\mathrm{sph}}

\newcommand{\reg}{\mathrm{reg}}

\newcommand{\F}{\mathrm{F}}
\newcommand{\G}{\mathrm{G}}

\newcommand{\R}{{\mathbb R}}
\newcommand{\N}{{\mathbb N}}
\newcommand{\C}{{\mathbb C}}

\newcommand{\OO}{{\mathbb O}}
\newcommand{\Pro}{{\mathbb P}}

\newcommand{\Z}{{\mathbb Z}}

\newcommand{\UH}{{\mathbb H}}
\newcommand{\Sph}{{\mathbb S}}

\numberwithin{equation}{section}

\makeindex

\begin{document}

\title{Matrix Valued Orthogonal Polynomials for Gelfand Pairs of Rank One}

\author
{Gert Heckman and Maarten van Pruijssen}

\date{\today}

\maketitle

\begin{abstract}
In this paper we study matrix valued orthogonal polynomials of one variable associated with
a compact connected Gelfand pair $(G,K)$ of rank one, as a generalization of earlier work
by Koornwinder \cite{Koornwinder} and subsequently by Koelink, van Pruijssen and Roman \cite{Koelink--van Pruijssen--Roman1},
\cite{Koelink--van Pruijssen--Roman2} for the pair $(\SU(2)\times\SU(2),\SU(2))$, and by
Gr\"unbaum, Pacharoni and Tirao \cite{Grunbaum--Pacharoni--Tirao1} for the pair $(\SU(3),\U(2))$.
Our method is based on representation theory using an explicit determination of the relevant branching rules.
Our matrix valued orthogonal polynomials have the Sturm--Liouville property of being eigenfunctions
of a second order matrix valued linear differential operator coming from the Casimir operator, and in fact
are eigenfunctions of a commutative algebra of matrix valued linear differential operators coming from $U(\mathfrak{g}_{c})^K$.
\end{abstract}

\newpage
\tableofcontents

\newpage

\section{Introduction}\label{intro}

For $N=1,2,3,\cdots$ a fixed positive integer let $\mathbb{M}$ denote the associative algebra
of square matrices of size $N\times N$ with complex entries.
Denote by $\mathbb{M}[x]$ the associative algebra of matrix valued polynomials.
A matrix valued weight function $W$ on some open interval $(a,b)$, with $-\infty\leq a<b\leq \infty$,
assigns to each $x\in(a,b)$ a selfadjoint matrix $W(x)\in\mathbb{M}$ (so $W(x)^{\dagger}=W(x)$),
which is positive definite (denoted $W(x)>0$) almost everywhere on $(a,b)$,
such that the matrix valued moments
\[ \int_a^b x^n W(x)dx \]
are finite (and selfadjoint) for all $n\in\N$.
Such a weight function defines a sesquilinear matrix valued form
\[ \langle P,Q\rangle= \int_a^b P^{\dagger}(x)W(x)Q(x)dx \]
on the polynomial algebra $\mathbb{M}[x]$.
Sesquilinear in the convention of this paper amounts to antilinear in the first and linear in the second argument.
The additional properties
\[ \langle PA,Q\rangle=A^{\dagger}\langle P,Q\rangle\;,\;\langle P,QA\rangle=\langle P,Q\rangle A\;,\;
   \langle P,Q\rangle^{\dagger}=\langle Q,P\rangle \]
for all $A\in\mathbb{M}$ and $P,Q\in\mathbb{M}[x]$ are trivially checked, while
\[ \langle P,P\rangle\geq0\;,\;\langle P,P\rangle=0\Leftrightarrow P=0 \]
holds for all $P\in\mathbb{M}[x]$, since $\{A\in\mathbb{M};A^{\dagger}=A,A\geq0\}$ is a convex cone,
and for $A$ in this cone $A=0\Leftrightarrow\tr{A}=0$.
Observe that $\langle P,P\rangle>0$ as soon as $\det P(x)\neq0$ at some point $x\in(a,b)$.

We can apply the Gram--Schmidt orthogonalization process to the
(right module for $\mathbb{M}$) basis $\{x^n;n\in\N\}$ of $\mathbb{M}[x]$.
By induction on $n$ we can define monic matrix valued polynomials $M_n(x)$ of degree $n$ by
\[ M_n(x)=x^n+\sum_{m=0}^{n-1}M_m(x)C_{n,m}\;,\;\langle M_m(x),x^n \rangle+\langle M_m(x),M_m(x)\rangle C_{n,m}=0 \]
for all $m<n$. Indeed, the matrix $C_{n,m}$ can be solved, because $\langle M_m,M_m\rangle>0$ and hence is invertible.
Since $\langle M_m,M_n\rangle=0$ for $m\neq n$ by construction any matrix valued polynomial $P(x)$ has a unique expansion
\[ P(x)=\sum_n M_n(x)C_n\;,\;\langle M_n,P \rangle=\langle M_n,M_n\rangle C_n \]
in terms of the basis $\{M_n;n\in\N\}$ of the monic orthogonal matrix valued polynomials.
The theory of matrix valued orthogonal polynomials was initiated by Krein \cite{Krein1}, \cite{Krein2},
and further developped by Geronimo \cite{Geronimo}, Duran \cite{Duran}, Gr\"{u}nbaum and Tirao \cite{Grunbaum--Tirao} and others.

In the scalar case $N=1$ with a non negative weight function $w(x)$ on the interval $(a,b)$
the system of monic orthogonal polynomials $p_n(x)$ has been the subject of an extensive study
in mathematical analysis over the past two centuries \cite{Szego}.
The classical orthogonal polynomials with weight functions
\[ w(x)=e^{-x^2/2}\;,\;w(x)=x^{\alpha}e^{-x}\;,\;w(x)=(1-x)^{\alpha}(1+x)^{\beta} \]
on the intervals $(-\infty,\infty)$, $(0,\infty)$, $(-1,1)$ for $\alpha,\beta>-1$
give rise to the Hermite, Laguerre and Jacobi polynomials respectively.
These three classes of orthogonal polynomials $p_n(x)$ are also eigenfunctions with eigenvalue $\lambda_n$
of a second order differential operator. Orthogonal polynomials with this additional
Sturm--Liouville property were characterized by Bochner \cite{Bochner},
who found besides the classical examples certain polynomials related to the Bessel function $J_{n+\frac12}(x)$.

In the matrix setting $N\geq1$ the question studied by Bochner was taken up by Duran \cite{Duran}
and further studied by Duran and Gr\"{u}nbaum \cite{Duran--Grunbaum}, and Gr\"{u}nbaum and Tirao \cite{Grunbaum--Tirao},
but a full list of matrix valued weight functions $W(x)$ with the Sturm--Liouville property
seems to be out of reach until now. Examples of matrix valued orthogonal polynomials
with the Sturm--Liouville property have been found using harmonic analysis for compact Gelfand pairs,
notably for the example $(\SU(2)\times\SU(2),\SU(2))$ (diagonally embedded) by Koornwinder \cite{Koornwinder}
and by Koelink, van Pruijssen and Rom\'{a}n \cite{Koelink--van Pruijssen--Roman1}, \cite{Koelink--van Pruijssen--Roman2},
and for the example $(\SU(3),\U(2))$ by Gr\"{u}nbaum, Pacheroni and Tirao
\cite{Grunbaum--Pacharoni--Tirao1}, \cite{Grunbaum--Pacharoni--Tirao2}.

The main goal of this paper is a uniform construction of a class of matrix valued orthogonal polynomials
with the Sturm--Liouville property, obtained using harmonic analysis for compact Lie groups.
More specifically, let $G$ be a compact connected Lie group, $K$ a closed connected subgroup
and $F$ a non empty face of the cone $P^+_K$ of dominant weights of $K$.
We say that $(G,K,F)$ is a multiplicity free system if for each irreducible representation $\pi^K_{\mu}$
of $K$ with highest weight $\mu\in F$ the induced representation $\mathrm{Ind}_K^G(\pi^K_{\mu})$
decomposes into a direct sum of irreducible representations $\pi^G_{\lambda}$ of G with highest weight $\lambda$,
with multiplicities
\[ m^{G,K}_{\lambda}(\mu)=[\pi^G_{\lambda}:\pi^K_{\mu}]\leq1 \]
for all $\lambda\in P_G^+$.
A necessary condition for $(G,K,F)$ to be a multiplicity free system is that the triple
$(G,K,\{0\})$ is multiplicity free, which is equivalent to $(G,K)$ being a Gelfand pair. 

Henceforth, in this paper we shall assume that $(G,K)$ is a Gelfand pair of rank one.
The classification of such pairs is known from the work of Kr{\"a}mer \cite{Kramer} and Brion \cite{Brion1}.
The space $G/K$ is either a sphere $\Sph^n$ or a projective space $\Pro^n(\F)$
with $n\geq2$ for $\F=\R,\C,\UH$ and $n=2$ for $\F=\OO$.
If $G$ is the maximal connected group of isometries, then $(G,K)$ is a symmetric pair of rank one.
In addition there are two exceptional spheres $\Sph^7=\Spin(7)/\mathrm{G}_2$ and $\Sph^6=\mathrm{G}_2/\SU(3)$,
which are acted upon in a distance transitive way, and so the corresponding pairs
$(G,K)$ are still Gelfand pairs of rank one.
The homogeneous spaces $G/K$ are precisely the distance regular spaces as found by Wang \cite{Wang}.
For $(G,K)$ a rank one Gelfand pair the classification of multiplicity free triples
$(G,K,F)$ is given by the following theorem.

\begin{theorem}\label{classification theorem}
The full list of multiplicity free rank one triples $(G,K,F)$ is given by the Table \ref{table: mfs}
\begin{table}
\begin{center}
\begin{tabular}{|l|l|l|l|r}
\hline
$G$ & $K$ & $\lambda_{\mathrm{sph}}$ & $\mathrm{faces}\;F$ \\ \hline
$\SU(n+1)$ & $\mathrm{S}(\mathrm{U}(n)\times\mathrm{U}(1))$ & $\varpi_1+\varpi_n$ & $\mathrm{any}$ \\ \hline
$\SO(2n+1)$ & $\SO(2n)$ & $\varpi_1$ & $\mathrm{any}$ \\ \hline
$\SO(2n)$ & $\SO(2n-1)$ & $\varpi_1$ & $\mathrm{any}$ \\ \hline
$\USp(2n)$ & $\USp(2n-2)\times\USp(2)$ & $\varpi_2$ & $\rk{F}\leq2$ \\ \hline
$\mathrm{F}_4$ & $\Spin(9)$ & $\varpi_1$ & $\rk F\le 1$ $\mathrm{or}$ \\
&&&$F=\N\omega_1+\N\omega_2$ \\ \hline
$\Spin(7)$ & $\mathrm{G}_2$ & $\varpi_3$ & $\rk{F}\leq1$ \\ \hline
$\mathrm{G}_2$ & $\SU(3)$ & $\varpi_1$ & $\rk{F}\leq1$ \\ \hline
\end{tabular}
\caption{Multiplicity free systems.}\label{table: mfs}
\end{center}
\end{table}
In the third column we have given the highest weight $\lambda_{\mathrm{sph}}\in P_G^+$ of the fundamental
zonal spherical representation in the notation for root systems of Knapp \cite{Knapp},
except for case $(G,K)=(\SO(4),\SO(3))$ that $G$ is not simple and $\lambda_{\mathrm{sph}}=\varpi_1+\varpi_2\in P_G^+=\N\varpi_1+\N\varpi_2$. Observe that $\lambda_{\mathrm{sph}}$ is a primitive vector in $P^{+}_{G}$.
\end{theorem}

The first three cases are well known through work of Weyl and Murnaghan \cite{Knapp}. In this paper we prove this theorem only in one direction, namely that all cases in the table give multiplicity free systems by working out the explicit branching rules in \S\S \ref{G2}, \ref{Spin7}, \ref{symplectic} and \ref{F4}. To exclude the case of the symplectic group with $\rk(F)\ge3$ we refer to \cite[Lem.~2.2.15]{van Pruijssen}, based on a result of Brion \cite[Prop.~3.1]{Brion1} or to \cite[Thm.~8.3]{He et al}.

The group $G$ for the two-point-homogeneous space $G/K$ admits a Cartan decomposition $G=KTK$ with $T\subset G$ a one dimensional torus with Lie algebra $\mathfrak{t}\subset\mathfrak{k}^{\perp}$. Denote $M=Z_{K}(T)$, the centralizer of $T$ in $K$. A triple $(G,K,F)$ is a multiplicity free system if and only if the restriction of $\pi^{K}_{\mu}$ to $M$ decomposes multiplicity free for all $\mu\in F$, which is proved in \cite[Prop.~2.2.9]{van Pruijssen} using the theory of spherical varieties. In the symmetric space examples this result goes back to Kostant and Camporesi \cite{Kostant, Camporesi1}.


For each of these triples $(G,K,F)$ we determine for all $\mu\in F$ the induced spectrum
\[ P^+_G(\mu)=\{\lambda\in P^+_G;m^{G,K}_{\lambda}(\mu)=1\} \]
explicitly through a case by case analysis.
We claim that if $\lambda\in P^+_G(\mu)$ then also $\lambda+\lambda_{\mathrm{sph}}\in P^+_G(\mu)$.
This can be derived from the Borel--Weil theorem.
Indeed, if $V^G_{\lambda}=H^0(G_c/B_c,L_{\lambda})$ denotes the Borel--Weil realization
of the finite dimensional representation of $G$ with highest weight $\lambda\in P_G^+$
then the intertwining projection
\[ V^G_{\lambda}\otimes V^G_{\lambda_{\mathrm{sph}}}\rightarrow V^G_{\lambda+\lambda_{\mathrm{sph}}} \]
onto the Cartan component of the tensor product is just realized
by the pointwise multiplication of holomorphic sections.

A spherical function of type $\mu\in F$ is a smooth map
$\Phi:G\rightarrow\End(V^K_{\mu})$ with the transformation rule
\begin{eqnarray}\label{trafo rule}
\Phi(kgk')=\pi^K_{\mu}(k)\Phi(g)\pi^K_{\mu}(k')
\end{eqnarray}
for all $g\in G$ and $k,k'\in K$.
The vector space $\mathcal{H}(G,K,\mu)$ of (say finite for $G$ on the left and the right)
spherical functions of type $\mu$ has a natural scalar valued Hermitian inner product
\[ \langle\Phi,\Phi'\rangle=\int_G\tr(\Phi(g)^{\dagger}\Phi'(g))dg \]
with the dagger coming from the (unique up to positive scalar) unitary structure on $V^K_{\mu}$, and $dg$ the normalized Haar measure on $G$.
Because $(G,K,F)$ is a multiplicity free system the elementary spherical functions
$\Phi^{\mu}_\lambda$ indexed by $\lambda\in P_G^+(\mu)$ form a basis for $\mathcal{H}(G,K,\mu)$,
which is orthogonal,
\[ \langle\Phi^{\mu}_{\lambda},\Phi^{\mu}_{\lambda'}\rangle=
   \frac{(\dim\mu)^2}{\dim\lambda}\delta_{\lambda,\lambda'}, \]
as a consequence of the Schur orthogonality relations.

With $\phi=\phi_{\mathrm{sph}}$ the fundamental zonal spherical function of $(G,K)$, the product $\phi\Phi^{\mu}_{\lambda}$ is again a spherical function of type $\mu$, and therefore has an expansion
\begin{eqnarray}\label{eqn: expansion}
 \phi\Phi^{\mu}_{\lambda}=\sum_{\lambda'}c_{\lambda,\lambda'}\Phi^{\mu}_{\lambda'}
\end{eqnarray}
with $\lambda'\in P_G^+(\mu)$.
For $\lambda,\lambda'\in P_G^+(\mu)$ the coefficient $c_{\lambda,\lambda'}=0$ unless
\begin{eqnarray}\label{eqn: inequality in well} \lambda-\lambda_{\mathrm{sph}}\preceq\lambda'\preceq\lambda+\lambda_{\mathrm{sph}},
\end{eqnarray}
where $\preceq$ is the usual partial ordering on $P^{+}_{G}$, and the leading coefficient $c_{\lambda,\lambda+\lambda_{\mathrm{sph}}}$ is non-zero. This allows one to define a degree $d:P^+_G(\mu)\rightarrow\N$ by
\[ d(\lambda+\lambda_{\mathrm{sph}})=d(\lambda)+1\;,\;\min\{d(P^+_G(\mu)\cap\{\lambda+\Z\lambda_{\mathrm{sph}}\})\}=0 \]
for all $\lambda\in P^+_G(\mu)$.

The bottom $B(\mu)$ of the induced spectrum $P_G^+(\mu)$ is defined as
\[ B(\mu)=\{\lambda\in P_G^+(\mu);d(\lambda)=0\} \]
giving $P_G^+(\mu)=B(\mu)+\N\lambda_{\mathrm{sph}}$ the structure of a well.
We have determined explicitly the structure of the bottom $B(\mu)$ with $\mu\in F$ for all
multiplicity free triples $(G,K,F)$ in the above table.
The first three lines of this table follow from a straightforward application of branching rules
going back to Weyl for the unitary group and Murnaghan for the orthogonal groups \cite{Knapp}.
The case of the symplectic group follows using the branching rule of Lepowsky \cite{Knapp,Lepowsky}, which under the
restriction $\rk{F}\leq2$ we are able to make completely explicit in \S\ref{symplectic}.
The remaining last two lines with the exceptional group of type $\mathrm{G}_2$ appearing turn out to be manageable as well and are treated in \S\S\ref{G2}, \ref{Spin7}.
The appropriate branching rules for the symmetric case $(\F_{4},\Spin(9))$ are calculated in \S \ref{F4}, using computer algebra.

Behind all these explicit calculations is a general multiplicity formula for branching rules going back to Kostant \cite{Lepowsky, Vogan} and rediscovered by Heckman \cite{Heckman}.
On the basis of our explicit knowledge of the bottom $B(\mu)$ for $\mu\in F$
we are able to verify case by case the following degree inequality in \S\S \ref{G2}, \ref{Spin7}, \ref{symplectic} and \ref{F4}.

\begin{theorem}\label{degree inequality theorem}
The degree $d:P^+_G(\mu)\rightarrow\N$ satisfies the inequality
\[ d(\lambda)-1\leq d(\lambda')\leq d(\lambda)+1 \]
for all $\lambda'\in P^+_G(\mu)$ with $c_{\lambda,\lambda'}\neq0$.
\end{theorem}

As stated before, in all cases of our table the restriction of $\pi^K_{\mu}$ for $\mu\in F$ to the centralizer $M$ of
a Cartan circle $T$ in $G$ is multiplicity free.
Moreover, the irreducible constituents are indexed in a natural way by the bottom $B(\mu)$, as we shall explain in \S\ref{mfs}.
The restriction of the elementary spherical function $\Phi^{\mu}_{\lambda}$ to the Cartan circle $T$
takes values in $\End_M(V^K_{\mu})$, and so is block diagonal by Schur's Lemma: $\End_M(V^K_{\mu})\cong\C^{N_{\mu}}$
with $N_{\mu}$ the cardinality of the bottom $B(\mu)$. Operators on the left become vectors on the right.
In view of this isomorphism, $\Phi^{\mu}_{\lambda}(t)$ for $t\in T$ is identified with the function $\Psi^{\mu}_{\lambda}(t)$ taking values in $\C^{N_{\mu}}$. We define for $n\in\N$ the matrix valued spherical functions $\Psi^{\mu}_n(t)$,
whose columns are the vector valued functions $\Psi^{\mu}_{\lambda}(t)$ with $\lambda\in P_G^+(\mu)$ of degree $d(\lambda)=n$.
Observe that both rows and columns of the matrix $\Psi^{\mu}_n(t)$ are indexed by the bottom $B(\mu)$.
Finally we can define our matrix valued polynomials $P^{\mu}_n(x)\in\mathbb{M}[x]$ of size $N_{\mu}\times N_{\mu}$ as functions of a real variable $x$ by
\[ \Psi^{\mu}_n(t)=\Psi^{\mu}_0(t)P^{\mu}_n(x) \]
with $t\mapsto x$ a new variable, defined by $x=c\phi+(1-c)$ for some $c>0$ (with $\phi$ the fundamental
zonal spherical function as before) in order to make the orthogonality interval $x(T)$ equal to $[-1,1]$.

The crucial fact that $P^{\mu}_n(x)$ is a matrix valued polynomial in $x$ of degree $n$ with invertible
leading coefficient $D^{\mu}_n$ (inductively given by $D^{\mu}_n=D^{\mu}_{n+1}A^{\mu}_n$) follows
from a three term recurrence relation
\[ xP^{\mu}_n(x)=P^{\mu}_{n+1}(x)A^{\mu}_n+P^{\mu}_{n}(x)B^{\mu}_n+P^{\mu}_{n-1}(x)C^{\mu}_n \]
which is obtained using the expansion (\ref{eqn: expansion}). 
Theorem \ref{degree inequality theorem} together with the ordering relation (\ref{eqn: inequality in well}) and $c_{\lambda,\lambda+\lambda_{\sph}}\ne0$ imply that the matrices $A_{n}$ are triangular with non-zero diagonal, and hence are invertible.
The matrix valued weight function is given by
\[ W^{\mu}(x)=(\Psi^{\mu}_0(t))^{\dagger}D^{\mu}\Psi^{\mu}_0(t)w(x) \]
with $w(x)=(1-x)^{\alpha}(1+x)^{\beta}$ the usual scalar weight function for the Cartan decomposition
$G=KTK$ and suitable $\alpha,\beta\in\N/2$ given in terms of root multiplicities.
The matrix $D^{\mu}$ is diagonal with entries the dimensions of the irreducible constituents of the restriction of $\pi^K_{\mu}$ to $M$,
which as a set was indexed by the bottom $B(\mu)$ as should. The diagonal matrix $D^{\mu}$ arises from the identification
$\End_M(V^K_{\mu})\cong\C^{N_{\mu}}$ with the trace form of the left operator side
and the standard Hermitian form on the right vector side.

The matrix valued polynomials $P^{\mu}_n(x)$ are orthogonal with respect to the
weight function $W^{\mu}(x)$ and have diagonal square norms, since
\[ \langle P^{\mu}_n,P^{\mu}_{n'}\rangle_{\nu,\nu'}=\langle\Phi^{\mu}_{\lambda},\Phi^{\mu}_{\lambda'}\rangle \]
with $\lambda=\nu+n\lambda_{\mathrm{sph}},\lambda'=\nu'+n'\lambda_{\mathrm{sph}}\in P^+_G(\mu)=B(\mu)+\N\lambda_{\mathrm{sph}}$.
Finally, the monic orthogonal polynomials $M^{\mu}_n(x)=x^n+\cdots$ and the orthogonal polynomials
$P^{\mu}_n(x)=M^{\mu}_n(x)D^{\mu}_n$ are related by eliminating the invertible leading coefficient $D^{\mu}_n$.

By Lie algebraic methods the polynomials $P^{\mu}_n(x)$ are shown to be eigenfunctions
of a commutative algebra $\mathbb{D}^{\mu}\subset\mathbb{M}[x,\del_{x}]$ of matrix valued differential operators
\[ DP^{\mu}_n=P^{\mu}_n\Lambda^{\mu}_n(D) \]
with $\Lambda^{\mu}_n(D)$ a diagonal eigenvalue matrix for all $D\in\mathbb{D}^{\mu}$.
The desired second order operator for the orthogonal polynomials with
the Sturm--Liouville property comes from the quadratic Casimir operator.
The dimension of the affine variety underlying the commutative algebra $\mathbb{D}^{\mu}$
is equal to the affine rank of the well $P^+_G(\mu)$.

Our explicit results on branching rules provide examples of the convexity theorem for
Hamiltonian actions of connected compact Lie groups on connected symplectic manifolds
with a proper moment map \cite{Heckman}, \cite{Guillemin--Sternberg2},
\cite{Guillemin--Sternberg3}, \cite{Guillemin--Sternberg4}, \cite{Kirwan}.
The multiplicities occur at the integral points in the moment polytopes
in accordance with the $[Q,R]=0$ principle of geometric quantization \cite{Guillemin--Sternberg1}.

In the next section we first discuss the pair $(G,K)=(G_2,\SU(3))$, which is an instructive
example to illustrate the various aspects of the representation theory and the construction
of the matrix valued orthogonal polynomials.

\textbf{Acknowledgement.} We thank Noud Aldenhoven for his help in programming certain branching rules, which gave us a good idea about the multiplicity freeness in the symplectic case. Furthermore, we thank Erik Koelink and Pablo Rom{\'a}n for fruitful discussions concerning matrix valued orthogonal polynomials.


\section{The pair $(G,K)=(\G_2,\SU(3))$}\label{G2}

In this section we take $G$ of type $\mathrm{G}_2$ and $K=\SU(3)$ the subgroup of type $A_2$.
Having the same rank the root systems $R_G$ of $G$ and $R_K$ of $K$ can be drawn in one picture,
and $R_K$ consists of the $6$ long roots.
The simple roots $\{\alpha_1,\alpha_2\}$ in $R_G^+$ and $\{\beta_1,\beta_2\}$ in $R_K^+$ are
indicated in Figure \ref{figure: roots for G2} and $P_G^+=\N\varpi_1+\N\varpi_2$ is contained in $P_K^+=\N\omega_1+\N\omega_2$.

\begin{figure}[ht]
\begin{center}
\begin{tikzpicture}[scale=1]


\pgfmathsetmacro\ax{2}
\pgfmathsetmacro\ay{0}
\pgfmathsetmacro\bx{2 * cos(120)}
\pgfmathsetmacro\by{2 * sin(120)}
\pgfmathsetmacro\lax{2*\ax/3 + \bx/3}
\pgfmathsetmacro\lay{2*\ay/3 + \by/3}
\pgfmathsetmacro\lbx{\ax/3 + 2*\bx/3}
\pgfmathsetmacro\lby{\ay/3 + 2*\by/3}


\begin{scope}

\clip (0,0) circle (3);

\foreach \k in {1,...,12} 
  \draw[dashed] (0,0) -- (\k * 30 + 30:40);

\foreach \a in {-10,...,-1,1,2,3,4,5}
   \draw[dotted] (\a*\lax-50*\lbx,\a*\lay-50*\lby) -- (\a*\lax+50*\lbx,\a*\lay+50*\lby);

\foreach \a in {-40,-39,...,40}
   \draw[dotted] (-50*\lax+\a*\lbx,-50*\lay+\a*\lby)-- (50*\lax+\a*\lbx,50*\lay+\a*\lby);

\foreach \a in {-40,-39,...,40}
   \draw[dotted] (-50*-\lax-50*\lbx+\a*\lbx,-50*-\lay-50*\lby+\a*\lby)-- (50*-\lax+50*\lbx+\a*\lbx,50*-\lay+50*\lby+\a*\lby);

\draw[thick,->] (0,0) -- (\ax,\ay) node[below] {\(\alpha_2=\beta_{2}\)};
\draw[thick,->] (0,0) -- (\bx,\by) node[left]{\(\beta_{1}\)};
\draw[thick,->] (0,0) -- (-\ax,-\ay); 
\draw[thick,->] (0,0) -- (-\bx,-\by); 
\draw[thick,->] (0,0) -- (\ax+\bx,\ay+\by) node[right] {\(\varpi_{2}\)};
\draw[thick,->] (0,0) -- (-\ax-\bx,-\ay-\by);

\draw[thick,->] (0,0) -- (\lax,\lay) node[right] {\(\omega_{2}\)};
\draw[thick,->] (0,0) -- (\lbx,\lby) node[above]{\(\varpi_{1}=\omega_{1}\)};
\draw[thick,->] (0,0) -- (-\lax,-\lay);
\draw[thick,->] (0,0) -- (-\lbx,-\lby);
\draw[thick,->] (0,0) -- (\lax-\lbx,\lay-\lby);
\draw[thick,->] (0,0) -- (\lbx-\lax,\lby-\lay) node[left] {\(\alpha_1\)};

\end{scope}
\end{tikzpicture}

\end{center}
\caption{Roots for $\G_{2}$.}\label{figure: roots for G2}
\end{figure}
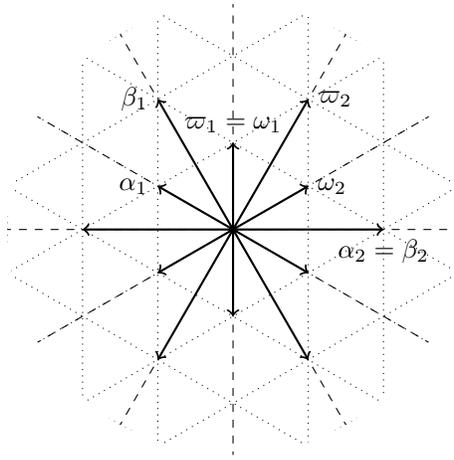

The branching rule from $G$ to $K$ is well known, see for example \cite{Heckman}.
In the picture below $s_1\in W_G$ is the orthogonal reflection in the mirror $\R\varpi_2$.
For $\lambda\in P_G^+$ the multiplicities $m_{\lambda}(\mu)$ for $\mu\in P_K^+$ are
supported in the gray region in the left picture. They have the familiar pattern of
the weight multiplicities for $\SU(3)$ as discussed in the various text books
\cite{Humphreys}, \cite{Fulton--Harris}. They are one on the outer hexagon, and increase by one
on each inner shell hexagon, untill the hexagon becomes a triangle, and from that moment on they stabilize.
Hence the restriction to $K$ of any irreducible representation of $G$ with highest weight $\lambda\in P_G^+$
is multiplicity free on the two rank one faces $\N\omega_1$ and $\N\omega_2$ of the dominant cone $P_K^+$.
In other words, the triples $(\mathrm{G}_2,\mathrm{A}_2,F_i=\N\omega_i)$ are multiplicity free for $i=1,2$,
which proves the last line of the table in Theorem{\;\ref{classification theorem}}.

The irreducible spherical representations of $G$ containing the trivial representation of $K$
have highest weight in $\N\varpi_1$, and $\lambda_{\mathrm{sph}}=\varpi_1$ is the fundamental
spherical weight. Given $\mu=n\omega_1\in F_1$ (and likewise $\mu=n\omega_2\in F_2$)
the corresponding induced spectrum of $G$ is multiplicity free by Frobenius reciprocity,
and by inversion of the branching rule has multiplicity one on the well shaped region
\[ P_G^+(\mu)=B(\mu)+\N\varpi_1\;,\;B(\mu)=\{k\varpi_1+l\varpi_2;k+l=n\}  \]
with bottom $B(\mu)$. The bottom is given by a single linear relation.

If we take $M$ the $\SU(2)$ group corresponding to the roots $\{\pm\alpha_2\}$
and denote by $p:P_G^+\rightarrow P_M^+=\N(\tfrac12\alpha_2)$ the natural projection
along the spherical direction $\varpi_1$, then $p$ is a bijection from the bottom $B(\mu)$
onto the image $p(B(\mu))$, which is just the restricted spectrum $P_M^+(\mu)$ for $M$
of the irreducible representation of $K$ with highest weight $\mu$.

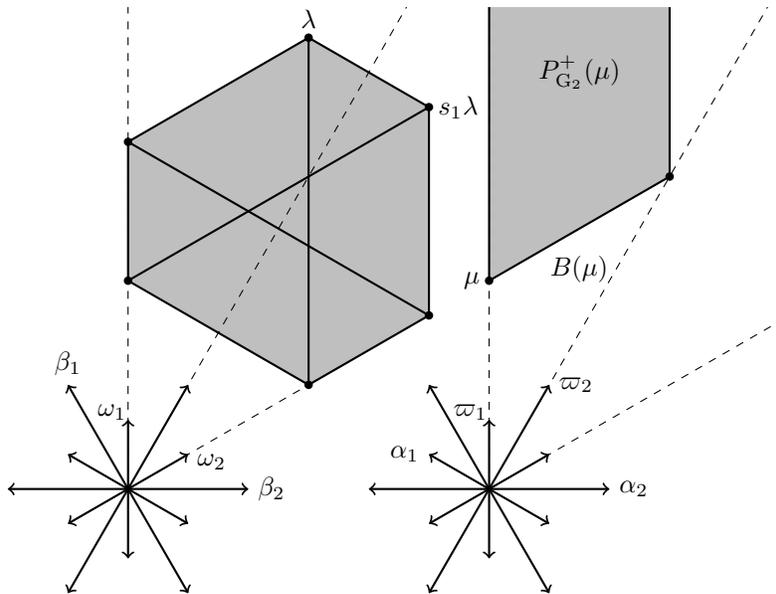
\begin{figure}[ht]
\begin{center}
\begin{tikzpicture}[scale=.8]


\pgfmathsetmacro\ax{2}
\pgfmathsetmacro\ay{0}
\pgfmathsetmacro\bx{2 * cos(120)}
\pgfmathsetmacro\by{2 * sin(120)}
\pgfmathsetmacro\lax{2*\ax/3 + \bx/3}
\pgfmathsetmacro\lay{2*\ay/3 + \by/3}
\pgfmathsetmacro\lbx{\ax/3 + 2*\bx/3}
\pgfmathsetmacro\lby{\ay/3 + 2*\by/3}


\begin{scope}

\clip (-2,-2) rectangle (3+4*\ax,8);

\draw[thick, fill=lightgray] (0,3*\lby) -- (0,5*\lby) -- (3*\lax+5*\lbx,3*\lay+5*\lby) -- (5*\lax,5*\lay+3*\lby) -- (5*\lax,5*\lay) -- (3*\lax,3*\lay) --cycle;

\draw[dashed] (0,0) -- (0,10);
\draw[dashed] (0,0)-- (8*\lax,8*\lay+8*\lby);
\draw[dashed] (0,0) -- (5*\lax,5*\lay);

\draw[thick,->] (0,0) -- (\ax,\ay) node[right] {\(\beta_{2}\)};
\draw[thick,->] (0,0) -- (\bx,\by) node[above] {\(\beta_{1}\)};
\draw[thick,->] (0,0) -- (-\ax,-\ay);
\draw[thick,->] (0,0) -- (-\bx,-\by);
\draw[thick,->] (0,0) -- (\ax+\bx,\ay+\by);
\draw[thick,->] (0,0) -- (-\ax-\bx,-\ay-\by);

\draw(\lax,\lay-1/8) node[right]{\(\omega_{2}\)};

\draw[thick,->] (0,0) -- (\lax,\lay);
\draw[thick,->] (0,0) -- (\lbx,\lby);
\draw (\lbx+1/8,\lby+1/8) node[left]{\(\omega_{1}\)};
\draw[thick,->] (0,0) -- (-\lax,-\lay);
\draw[thick,->] (0,0) -- (-\lbx,-\lby);
\draw[thick,->] (0,0) -- (\lax-\lbx,\lay-\lby);
\draw[thick,->] (0,0) -- (\lbx-\lax,\lby-\lay);

\fill (0,3*\lby) circle (2pt);
\fill (0,5*\lby) circle (2pt);
\fill (3*\lax+5*\lbx,3*\lay+5*\lby)circle (2pt) node[above] {\(\lambda\)};
\fill (5*\lax,5*\lay+3*\lby)circle (2pt) node[right] {\(s_{1}\lambda\)};
\fill (5*\lax,5*\lay) circle (2pt);
\fill (3*\lax,3*\lay) circle (2pt);

\draw[thick] (0,3*\lby) -- (5*\lax,5*\lay+3*\lby);
\draw[thick] (0,5*\lby) -- (5*\lax,5*\lay);
\draw[thick] (3*\lax+5*\lbx,3*\lay+5*\lby) -- (3*\lax,3*\lay);

\draw[thick, fill=lightgray] (3*\ax,3*\lby) -- (3*\ax,9*\lby) -- (3*\lax+3*\ax,3*\lay+9*\lby) -- (3*\lax+3*\ax,3*\lay+3*\lby) -- cycle;

\draw[dashed] (0+3*\ax,0) -- (0+3*\ax,10);
\draw[dashed] (0+3*\ax,0)-- (15*\lax+3*\ax,15*\lay+15*\lby);
\draw[dashed] (0+3*\ax,0) -- (15*\lax+3*\ax,15*\lay);

\draw[thick,->] (0+3*\ax,0) -- (\ax+3*\ax,\ay) node[right] {\(\alpha_{2}\)};
\draw[thick,->] (0+3*\ax,0) -- (\bx+3*\ax,\by);
\draw[thick,->] (0+3*\ax,0) -- (-\ax+3*\ax,-\ay);
\draw[thick,->] (0+3*\ax,0) -- (-\bx+3*\ax,-\by);
\draw[thick,->] (0+3*\ax,0) -- (\ax+\bx+3*\ax,\ay+\by) node[right] {\(\varpi_{2}\)};
\draw[thick,->] (0+3*\ax,0) -- (-\ax-\bx+3*\ax,-\ay-\by);

\draw(\lax+3*\ax,\lay-1/8);
\draw[thick,->] (0+3*\ax,0) -- (\lax+3*\ax,\lay);
\draw[thick,->] (0+3*\ax,0) -- (\lbx+3*\ax,\lby);
\draw (\lbx+3*\ax+1/8,\lby+1/8) node[left]{\(\varpi_{1}\)};
\draw[thick,->] (0+3*\ax,0) -- (-\lax+3*\ax,-\lay);
\draw[thick,->] (0+3*\ax,0) -- (-\lbx+3*\ax,-\lby);
\draw[thick,->] (0+3*\ax,0) -- (\lax-\lbx+3*\ax,\lay-\lby);
\draw[thick,->] (0+3*\ax,0) -- (\lbx-\lax+3*\ax,\lby-\lay) node[left] {\(\alpha_{1}\)};

\fill (3*\ax,3*\lby) circle (2pt) node[left] {\(\mu\)};
\fill (3*\lax+3*\lbx+3*\ax,3*\lay+3*\lby) circle (2pt);
\draw (3/2*\lax+3*\ax,6*\lby) node {\(P_{\mathrm{G}_{2}}^{+}(\mu)\)};
\draw (3/2*\lax+3*\ax,7/2*\lby) node[below] {\(B(\mu)\)};


\end{scope}
\end{tikzpicture}
\end{center}
\caption{Branching from $\G_{2}$ to $\SU(3)$ on the left and the $\mu$-well on the right.}\label{intro: figure G2A2 branching}
\end{figure}

There is warning about the choice of the various Cartan subalgebras.
In order to compute branching rules it is natural and convenient (as we did above)
to choose the Cartan subalgebra of $K$ contained in the Cartan subalgebra of $G$.
The other choice is that we start with a rank one Gelfand pair $(G,K)$,
and choose the Cartan circle group $T$ in $G$ perpendicular to $K$.
If $M$ is the centralizer of $T$ in $K$, then $MT$ is a subgroup in $G$ of full rank.
A maximal torus in $MT$ is then a maximal torus for $G$ as well.
But this maximal torus need not contain a maximal torus for $K$,
as is clear from the present example. It will only do so if the rank of $K$
is equal to the rank of $M$, which is equal to the rank of $G$ minus $1$,
and a maximal torus of $M$ is a maximal torus of $K$ as well.


\section{Multiplicity free systems}\label{mfs}

Connected compact irreducible Gelfand pairs $(G,K)$ have been classified by Kr\"{a}mer
for $G$ a simple Lie group and by Brion for $G$ a semisimple Lie group \cite{Kramer}, \cite{Brion1}.
We shall assume that $G$ and $K$ are connected, and that the connected space $G/K$ is also simply connected.
The pair $(G,K)$ is called rank one if the Hecke algebra $\mathcal{H}(G,K)$ of zonal  spherical
(so bi-$G$-finite and bi-$K$-invariant) functions is a polynomial algebra $\C[\phi]$ with one generator,
the fundamental elementary zonal spherical function $\phi=\phi_{\mathrm{sph}}$.
We shall assume throughout this paper that $(G,K)$ is a rank one Gelfand pair, with $G/K$ simply connected.
The corresponding spaces $G/K$ are just the distance regular spaces found by Wang \cite{Wang}.

Indeed, for $K<G$ compact connected Lie groups the homogeneous space $G/K$ equipped with an invariant Riemannian metric
is distance transitive for the action of $G$ on $G/K$ if and only if the action of $K$ on the tangent space $T_{eK}G/K$
is transitive on the unit sphere. This is equivalent with the algebra $P(T_{eK}G/K)^K$ of polynomial invariants being
a polynomial algebra in a single generator (the quadratic norm), which in turn is equivalent with the Hecke algebra
$\mathcal{H}(G,K)$ being a polynomial algebra $\C[\phi]$ in the single generator $\phi=\phi_{\mathrm{sph}}$.
If $\mathcal{H}(G,K)=\C[\phi]$ has a single generator then it is commutative as convolution algebra,
which is equivalent with $(G,K)$ being a Gelfand pair.

Let $\mathfrak{k}<\mathfrak{g}$ be the Lie algebras of $K<G$.
By definition the infinitesimal Cartan decomposition $\mathfrak{g}=\mathfrak{k}\oplus\mathfrak{p}$
is the orthogonal decomposition with respect to minus the Killing form on $\mathfrak{g}$.
Since $(G,K)$ has rank one the adjoint homomorphism $K\rightarrow\SO(\mathfrak{p})$ is a surjection.
Fix a (maximal Abelian) one dimensional subspace $\mathfrak{t}$ in $\mathfrak{p}$.
Any two such are clearly conjugated by $K$, and let $T<G$ be the corresponding Cartan circle group.
Let $M<N$ be the centralizer and normalizer of $T$ in $K$ with Lie algebra $\mathfrak{m}$.
The Weyl group $W=N/M$ has order $2$ and acts on $T$ by $t\mapsto t^{\pm1}$.

The subgroup $MT$ has maximal rank in $G$, and choosing a maximal torus in $MT$ for $G$ defines a natural
restriction map from the weight lattice $P_G$ of $G$ to the weight lattice of the circle $T$.
The next result for symmetric pairs is just the Cartan--Helgason theorem.

\begin{proposition}\label{Cartan circle group}
Suppose $G$ is simply connected and $K$ is connected, so that $G/K$ is simply connected.
Then $T\cap K$ has order $2$, except for the Gelfand pair $(G,K)=(\Spin(7),G_2)$ where it has order $3$.
\end{proposition}

\begin{proof}
The crucial remark is that the highest weight $\lambda_{\mathrm{sph}}\in P_G^+$ of the fundamental zonal spherical
representation of $(G,K)$ after restriction to $T$ becomes a generator for the weight lattice of $T/(T\cap K)$.
For a symmetric pair $(G,K)$ with Cartan involution $\theta:G\rightarrow G$ we have $K=G^{\theta}$ and $\theta(t)=t^{-1}$ for $t\in T$.
Hence $T\cap K$ has order $2$ for $(G,K)$ a symmetric pair. In the remaining two cases we use the notation of Bourbaki \cite{Bourbaki}.

For $(G,K)=(\Spin(7),G_2)$ the weight lattice of $G$ is naturally identified with $\Z^3$
with basis $\epsilon_i$, and likewise the dual coroot lattice becomes $\Z^3$ with basis $e_i$.
The character lattice of $T/(T\cap K)$ has generator $\varpi_3=(\epsilon_1+\epsilon_2+\epsilon_3)/2$,
which takes the value $3$ on the generator $2(e_1+e_2+e_3)$ of the coroot lattice of $T$.

For $(G,K)=(G_2,\SU(3))$ the weight lattice of $G$ is naturally identified with
$\{\xi\in\Z^3;\xi_1+\xi_2+\xi_3=0\}$, and likewise the dual coroot lattice becomes $\{x\in\Z^3;x_1+x_2+x_3=0\}$.
The character lattice of $T/(T\cap K)$ has generator $\varpi_1=2\alpha_1+\alpha_2=-\epsilon_2+\epsilon_3$,
which takes the value $2$ on the generator $-e_2+e_3$ of the coroot lattice of $T$.
\end{proof}

In the next definition we explain the well shape of the induced spectrum
$P_G^+(\mu)=B(\mu)+\N\lambda_{\mathrm{sph}}$ with bottom $B(\mu)$.
This idea goes back to Kostant and Camporesi \cite{Kostant}, \cite{Camporesi1}.

\begin{definition}
For $\mu\in P_K^+$ the highest weight of an irreducible representation $\pi^K_{\mu}$ of $K$ the
induced representation $\mathrm{Ind}^G_K(\pi^K_{\mu})$ decomposes as a direct sum of irreducible
representations $\pi^G_{\lambda}$ of $G$ with branching multiplicities
\[ m^{G,K}_{\lambda}(\mu)=[\pi^G_{\lambda}:\pi^K_{\mu}] \]
for all $\lambda\in P_G^+$ by Frobenius reciprocity. We denote
\[ P_G^+(\mu)=\{\lambda\in P_G^+;m^{G,K}_{\lambda}(\mu)\geq1\} \]
for the induced spectrum. In the introduction we have explained using the Borel--Weil theorem that $\lambda\in P_G^+(\mu)$
implies $\lambda+\lambda_{\mathrm{sph}}\in P_G^+(\mu)$. In turn we see that $P_G^+(\mu)=B(\mu)+\N\lambda_{\mathrm{sph}}$
has the shape of a well with
\[ B(\mu)=\{\lambda\in P_G^+(\mu);\lambda-\lambda_{\mathrm{sph}}\notin P_G^+(\mu)\} \]
the bottom of the induced spectrum $P_G^+(\mu)$.
\end{definition}

To arrive at a good theory of matrix valued orthogonal polynomials we have to restrict ourselves
to multiplicity free triples $(G,K,\mu)$ and $(G,K,F)$ for $\mu\in P_K^+$ a suitable dominant weight for $K$
and $F$ a suitable facet of the dominant cone $P_K^+$ for $K$.

\begin{definition}
The triple $(G,K,\mu)$ with $\mu\in P_K^+$ a highest weight for $K$ is called multiplicity free if
the branching multiplicity $m_{\lambda}(\mu)\leq1$ for all $\lambda\in P_G^+$,
so if the induced representation $\mathrm{Ind}^G_K(\pi^K_{\mu})$ decomposes multiplicity free as a representation of $G$.
Likewise, $(G,K,F)$ is called a multiplicity free system with $F$ a facet of the dominant integral cone $P_K^+$ if $(G,K,\mu)$ is multiplicity free for all $\mu\in F$.
\end{definition}

Camporesi calculated the bottoms $B(\mu)$ of the well $P_G^+(\mu)$ explicitly in the first three examples of the table
in Theorem{\;\ref{classification theorem}} using the classical branching laws of Weyl for the unitary group
and Murnaghan for the orthogonal groups \cite{Camporesi1},\cite{Knapp}. In the fourth example of the symplectic group
he obtained partial results, because of the complexity of the branching law of Lepowsky (from $G$ to $K$)
\cite{Lepowsky} and of Baldoni Silva (from $K$ to $M$) \cite{Baldoni Silva} in that case.
However, in this symplectic case the restriction on a multiplicity free system$(G,K,F)$ is
just strong enough to find a completely explicit description of the bottom.

\begin{proposition}\label{prop: reducing to M}
Let $F$ be a facet of the dominant integral cone $P_K^+$.
Then the branching multiplicity $m^{G,K}_{\lambda}(\mu)\leq1$ for all $\mu\in F$ and all dominant weights $\lambda\in P_G^+$ if and only if the branching
multiplicity $m^{K,M}_{\mu}(\nu)\leq1$ for all $\mu\in F$ and
all dominant weights $\nu\in P_M^+$.
\end{proposition}

\begin{proof}
Let us complexify all our compact Lie groups $G,K,M,T$ to complex reductive algebraic groups $G_c,K_c,M_c,T_c$.
The statement of the proposition translates into the following geometric statement.
For $P_c$ the parabolic subgroup of $K_c$ with Levi component the stabilizer of $F$
the variety $G_c/P_c$ is spherical for $G_c$ if and only if $K_c/P_c$ is spherical for $M_c$.
Here we say that a variety with an action of a reductive group is spherical if the Borel subgroup has an open orbit.
Observe that (also for the non symmetric pairs) we have an infinitesimal Iwasawa decomposition
\[ \mathfrak{g}_c=\mathfrak{k}_c\oplus\mathfrak{t}_c\oplus\mathfrak{n}_c \]
with $\mathfrak{n}_c$ the direct sum of those root spaces $\mathfrak{g}_c^{\alpha}$ for which the restriction of
$\alpha$ to $\mathfrak{t}_c$ is a positive multiple of the restriction of $\lambda_{\mathrm{sph}}$ to $\mathfrak{t}_c$.
Taking the Borel subgroup of $G_c$ of the form $B_{M_{c}}T_cN_c$ with $B_{M_{c}}$ a Borel subgroup for $M_c$
the equivalence of $G_c/P_c$ having an open orbit for $B_{M_{c}}T_cN_c$ is equivalent to $K_c/P_c$
having an open orbit for $B_{M_{c}}$ follows, since the orbit of $T_cN_c$ through $K_c$ is open in $G_c/K_c$.
\end{proof}

Let us take the Cartan subalgebra of $\mathfrak{g}_c$ a direct sum of $\mathfrak{t}_c$
and a Cartan subalgebra of $\mathfrak{m}_c$, and extend a set of positive roots for
$\mathfrak{m}_c$ to a set of positive roots for $\mathfrak{g}_c$.
Let $V^G_{\lambda}$ be an irreducible representation of $G$ with highest weight $\lambda\in P_G^+$.
Because $M_cT_cN_c$ is a standard parabolic subgroup of $G_c$ the vector space
\[ (V^G_{\lambda})^{\mathfrak{n}_c}=\{v\in V^G_{\lambda};Xv=0\;\forall X\in\mathfrak{n}_c\} \]
is an irreducible representation of $M_{c}$ with highest weight $\nu\in P_M^+$.
Clearly $\nu=p(\lambda)$ with $p:P_G^+\rightarrow P_M^+$ the natural projection
along the spherical direction $\N\lambda_{\mathrm{sph}}$.
The Iwasawa decomposition
$\mathfrak{g}_c=\mathfrak{k}_c\oplus\mathfrak{t}_c\oplus\mathfrak{n}_c$ of the above proof gives the
Poincar\'{e}--Birkhoff-Witt factorization $U(\mathfrak{g}_c)=U(\mathfrak{k}_c)U(\mathfrak{t}_c)U(\mathfrak{n}_c)$
and we conclude that $U(\mathfrak{k}_c)(V^G_{\lambda})^{\mathfrak{n}_c}=V^G_{\lambda}$.

\begin{proposition}\label{projection from induced to restricted spectrum}
Let $(G,K,F)$ be a multiplicity free system and let $\mu\in F$.
Then the natural projection $p:P_G^+\rightarrow P_M^+$ is a surjection
from the induced spectrum $P_G^+(\mu)$ for $G$ onto the restricted spectrum
\[ P_M^+(\mu)=\{\nu\in P_M^+;m^{K,M}_{\mu}(\nu)\geq1\} \]
for $M$, and thefore $p:B(\mu)\rightarrow P_M^+(\mu)$ is a bijection.
Note that $m^{K,M}_{\mu}(\nu)\leq1$ for all $\nu\in P_M^+$ by the previous proposition.
\end{proposition}

\begin{proof}
Let $\langle\cdot,\cdot\rangle$ be a unitary structure on $V^G_{\lambda}$ for $G$.
Suppose $V$ is an irreducible subrepresentation of $K$ in the restriction of $V^G_{\lambda}$ to $K$.
If $u$ is a nonzero vector in $(V^G_{\lambda})^{\mathfrak{n}_c}$ then $\langle u,v\rangle\neq0$
for some $v\in V$. Indeed $\langle u,v\rangle=0$ for all $v\in V$ contradicts
$U(\mathfrak{k}_c)(V^G_{\lambda})^{\mathfrak{n}_c}=V^G_{\lambda}$.
Hence the restriction of $V$ to $M$ contains a copy of $V^M_{p(\lambda)}$ by Schur's Lemma.
One of the subspaces $V$ is a copy of $V^K_{\mu}$, and so $m_{\mu}(p(\lambda))\geq1$.
This proves that the natural projection $p$ maps the induced spectrum $P_G^+(\mu)$ of $G$
inside the restricted spectrum $P_M^+(\mu)$ of $M$.

It remains to show that
\[ p:P_G^+(\mu)\rightarrow P_M^+(\mu) \]
is onto for all $\mu\in F$. This follows from Proposition \ref{prop: asymptotics}.
\end{proof}


\begin{proposition}\label{prop: asymptotics}
Let $\lambda\in P^{+}_{G}$, $\mu\in P^{+}_{K}$ and let $p:P_G^+\rightarrow P_M^+$ be the natural projection. Then
\begin{itemize}
\item $m^{G,K}_{\lambda+k\lambda_{\sph}}(\mu)\le m^{G,K}_{\lambda+s\lambda_{\sph}}(\mu)$ if $k\le s$ and
\item $\lim_{n\to\infty}m^{G,K}_{\lambda+n\lambda_{\sph}}(\mu)=m^{K,M}_{\mu}(p(\lambda)).$
\end{itemize}
\end{proposition}

\begin{proof} Every irreducible $K$-representation that occurs in the $K$-module $V_{\lambda}$ also occurs in the $K$-module $V_{\lambda+\lambda_{\sph}}$. Indeed, let $v_{K}\in V_{\lambda_{\sph}}$ be a non-zero $K$-fixed vector and consider the composition of $V_{\lambda}\to V_{\lambda}\otimes V_{\lambda_{\sph}}:v\mapsto v\otimes v_{K}$ and the projection $V_{\lambda}\otimes V_{\lambda_{\sph}}\to V_{\lambda+\lambda_{\sph}}$. Both maps intertwine the $K$-action and the first statement follows.

For $(G,K)$ a symmetric pair (even of arbitrary rank)
the second statement is a result of Kostant \cite[Thm.~3.5]{Kostant} and Wallach \cite[Cor.~8.5.15]{Wallach}. For spherical pairs $(G,K)$ a similar stability result is shown by Kitagawa \cite[Cor.~4.10]{Kitagawa}. However, since we have control over the branching rules of the remaining non-symmetric pairs, we present our own proof. 

Consider the triple $(G,K,M)=(\G_2,\SU(3),\SU(2))$. Let $\lambda=n_1\varpi_1+n_2\varpi_2\in P_G^+$ with $n_1$ relatively large, $\mu=m_1\omega_1+m_2\omega_2\in P_K^+$ and $\nu=n_2p(\varpi_2)\in P_M^+$. 
On the one hand, we find 
\[ m^{G,K}_{\lambda}(\mu)=\min\{m_1+1,m_2+1,m_1+m_2-n_2+1,n_2+1\} \]
as in clear from the left side of Figure \ref{intro: figure G2A2 branching}. Indeed $m_1+1$ comes from the disctance
of $\mu$ to the face $\N\omega_1$, and similarly $m_2+1$ for the face $\N\omega_2$. 
The expression $m_1+m_2-n_2+1$ comes from the middle linear constraint,
while $n_2+1$ comes from the middle truncation. 
The other three constraints disappear as $n_1$ gets large. 
On the other hand we get 
\[ m^{K,M}_{\mu}(\nu)=\min\{n_2+1,\min\{m_1,m_2\}+1,m_1+m_2-n_2+1\} \]
as easily checked from the familiar branching from $\SU(3)$ to $\SU(2)$. It follows that $m^{G,K}_{\lambda}(\mu)=m^{K,M}_{\mu}(\nu)$ for large $n_{1}$.

The proof for the case $(\Spin(7),\G_{2},\SU(3))$ is postponed to the end of Section \ref{Spin7}, where we discuss the branching rules that are needed.
\end{proof}


\section{The pair $(G,K)=(\Spin(7),\mathrm{G}_2)$}\label{Spin7}

In this section we take $G=\Spin(7)$ with complexified Lie algebra $\mathfrak{g}$ of type $\mathrm{B}_3$.
Let $\mathfrak{t}_G\cong\C^3$ be a Cartan subalgebra with positive roots $R_G^+$ given by
\[ e_i-e_j,\;e_i+e_j,\;e_i \]
for $1\leq i<j\leq3$, and basis of simple roots $\alpha_1=e_1-e_2,\alpha_2=e_2-e_3,\alpha_3=e_3$.
The fundamental weights $\varpi_1=e_1,\varpi_2=e_1+e_2,\varpi_3=(e_1+e_2+e_3)/2$ are a basis
over $\N$ for the cone $P_G^+$ of dominant weights.

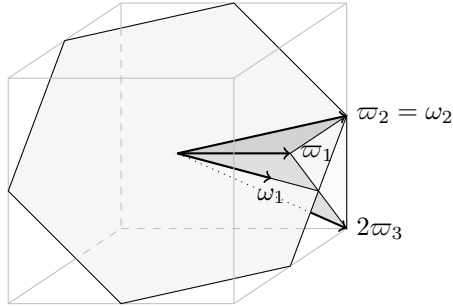
\begin{figure}[ht]
\begin{center}
\begin{tikzpicture}[scale=.5]
\begin{scope}

\clip (-6,-6) rectangle (8,6);


\draw[thin,lightgray] (9/2,-2) -- (3/2,-4);
\draw[thin,lightgray] (-3/2,4) -- (-9/2,2);
\draw[thin,lightgray] (-9/2,-4) -- (-27/10,-14/5);
\draw[thin,lightgray] (3/2,-4) -- (-9/2,-4) -- (-9/2,2);
\draw[thin,lightgray] (-3/2,4) -- (9/2,4);
\draw[thin,lightgray] (9/2,4) -- (9/2,-2);
\draw[thin,lightgray] (9/2,-2) -- (27/8,-2);

\draw[fill, lightgray!15] (0,0) -- (15/4,-1) -- (3,-3) -- (-3/2,-4) -- (-9/2,-1) -- (-3,3) -- (3/2,4) -- (9/2,1) -- cycle;
\draw[thin] (99/28,-11/7) -- (3,-3) -- (-3/2,-4) -- (-9/2,-1) -- (-3,3) -- (3/2,4) -- (9/2,1) ;

\draw[fill,gray!35] (0,0) -- (9/2,1) -- (3,0) -- cycle;

\draw[fill,gray!25] (0,0) -- (3,0) -- (15/4,-1) -- cycle;
\draw[fill,gray!25] (15/4,-1) -- (9/2,-2) -- (99/28,-11/7) -- cycle;

\def\mypath{(15/4,-1) -- (3,0) -- (9/2,1) --cycle}
\fill [lightgray!15] \mypath;

\draw[fill,lightgray!5] (15/4,-1) -- (9/2,1) -- (9/2,-2) --cycle;
\draw[thin] (15/4,-1) -- (3,0) -- (9/2,1) --cycle;

\draw[thin] (0,0) -- (15/4,-1);
\draw[thick,->] (0,0) -- (3,0) node[right] {\(\varpi_{1}\)};
\draw[thick,->] (0,0) -- (9/2,1) node[right] {\(\varpi_{2}=\omega_{2}\)};
\draw[thick,->] (0,0) -- (5/2,-2/3) node[below] {\(\omega_{1}\)};
\draw[thin] (9/2,1) -- (9/2,-2) -- (15/4,-1) -- (99/28,-11/7);

\draw[dotted] (0,0) -- (99/28,-11/7);
\draw[thick,->] (99/28,-11/7) -- (9/2,-2) node[right]{\(2\varpi_{3}\)};

\draw[thin,lightgray] (9/2,4) -- (3/2,2);
\draw[thin,lightgray] (-9/2,2) -- (3/2,2) -- (3/2,-4);

\draw[thin,lightgray] (-3/2,4) -- (-3/2,10/3);

\draw[dashed,lightgray] (-3/2,10/3) -- (-3/2,-2) -- (-27/10,-14/5);

\draw[dashed,lightgray] (-3/2,-2) -- (17/5,-2);

\end{scope}

\end{tikzpicture}
\end{center}
\caption{Fundamental weights for $\Spin(7)$ and $\G_{2}$.}

\end{figure}

As the Cartan subalgebra $\mathfrak{t}_K$ for $K=\mathrm{G}_2$ we shall take the orthogonal complement of $h=(-e_1+e_2+e_3)$.
The elements $e_1+e_3,e_1+e_2,e_2-e_3$ are the long positive roots in $R_K^+$, while
\[ \epsilon_1=(2e_1+e_2+e_3)/3,\;\epsilon_2=(e_1+2e_2-e_3)/3,\;\epsilon_3=(e_1-e_2+2e_3)/3 \]
are the short positive roots in $R_K^+$. The natural projection $q:R_G^+\rightarrow P_K^+$
is a bijection onto the long roots and two to one onto the short roots in $R_K^+$.
Note that $\epsilon_i=q(e_i)$ for $i=1,2,3$.
The simple roots in $R_K^+$ are $\{\beta_1=\epsilon_3,\beta_2=\epsilon_2-\epsilon_3\}$ with corresponding
fundamental weights $\{\omega_1=\epsilon_1,\omega_2=\epsilon_1+\epsilon_2\}$.
Observe that $\omega_1=q(\varpi_1)=q(\varpi_3)$ and $\omega_2=q(\varpi_2)$,
and hence $q:P_G^+\rightarrow P_K^+$ is a surjection.
Note that the natural projection $q:P_G\rightarrow P_K$ is equivariant for the action of
the Weyl group $W_M\cong\mathfrak{S}_3$ of the centralizer $M=\SU(3)$ in $K$ of $h$.
The Weyl group $W_G$ is the semidirect product of $C_2\times C_2\times C_2$
acting by sign changes on the three coordinates and the permutation group $\mathfrak{S}_3$.

As a set with multiplicities we have
\[ A=q(R_G^+)-R_K^+=\{\epsilon_1,\epsilon_2,\epsilon_3\} \]
whose partition function $p_A$ enters in the formula for the branching from $\mathrm{B}_3$ to $\mathrm{G}_2$.
Note that $p_A(k\epsilon_1+l\epsilon_2)=p_A(k\epsilon_1+m\epsilon_3)=k+1$ for $k,l,m\in \N$
and $p_A(\mu)=0$ otherwise.

\begin{lemma}\label{branching rule lemma}
For $\lambda\in P_G^+$ and $\mu\in P_K^+$ the multiplicity $m_{\lambda}^{G,K}(\mu)\in \N$
with which an irreducible representation of $K$ with highest weight $\mu$ occurs in the
restriction to $K$ of an irreducible representation of $G$ with highest weight $\lambda$
is given by
\[ m_{\lambda}^{G,K}(\mu)=\sum_{w\in W_G}\det(w)p_A(q(w(\lambda+\rho_G)-\rho_G)-\mu) \]
and if we extend $m_{\lambda}^{G,K}(\mu)\in \Z$ by this formula for all $\lambda\in P_G$ and $\mu\in P_K$ then
\[ m_{w(\lambda+\rho_G)-\rho_G}^{G,K}(v(\mu+\rho_K)-\rho_K)=\det(w)\det(v)m_{\lambda}^{G,K}(\mu) \]
for all $w\in W_G$ and $v\in W_K$. Here $\rho_G$ and $\rho_K$ are the Weyl vectors of $R_G^+$ and $R_K^+$ respectively.
\end{lemma}

This lemma was obtained in Heckman \cite{Heckman} as a direct application of the Weyl charcter formula.
The above type formula, valid for any pair $K<G$ of connected compact Lie groups \cite{Heckman},
might be cumbersome for practical computations of the multiplicities, because of the (possibly large) alternating
sum over a Weyl group $W_G$ and the piecewise polynomial behaviour of the partition function.
However in the present (fairly small) example one can proceed as follows.

If $\lambda=k\varpi_1+l\varpi_2+m\varpi_3=klm=(x,y,z)$ with
\[ x=k+l+m/2,y=l+m/2,z=m/2\Leftrightarrow k=x-y,l=y-z,m=2z \]
then $\lambda$ is dominant if $k,l,m\geq0$ or equivalently $x\geq y\geq z\geq0$.
We tabulate the $8$ elements $w_1,\cdots,w_8\in W_G$ such that the projection
$q(w_i\lambda)\in\N\epsilon_1+\N\epsilon_2$ is dominant for $R_M^+$ for all $\lambda$ which are dominant for $R_G^+$.
Clearly the projection of $(x,y,z)$ is given by
\[ q(x,y,z)=x\epsilon_1+y\epsilon_2+z\epsilon_3=(x+z)\epsilon_1+(y-z)\epsilon_2 \]
and $\rho_G=\varpi_1+\varpi_2+\varpi_3=(2\tfrac12,1\tfrac12,\tfrac12)$ is the Weyl vector for $R_G^+$.

\begin{table}[ht]
\begin{center}
\begin{tabular}{|l|l|l|l|l|} \hline
$i$&$\det(w_i)$&$w_i\lambda$&$q(w_i\lambda)$&$q(w_i\rho_G-\rho_G)$\\ \hline
$1$&$+$&$(x,y,z)$&$(x+z)\epsilon_1+(y-z)\epsilon_2$&$0$\\ \hline
$2$&$-$&$(x,y,-z)$&$(x-z)\epsilon_1+(y+z)\epsilon_2$&$-\epsilon_3$\\ \hline
$3$&$+$&$(x,z,-y)$&$(x-y)\epsilon_1+(y+z)\epsilon_2$&$-\epsilon_1-\epsilon_3$\\ \hline
$4$&$-$&$(x,-z,-y)$&$(x-y)\epsilon_1+(y-z)\epsilon_2$&$-\epsilon_1-\epsilon_2-\epsilon_3$\\ \hline
$5$&$-$&$(y,x,z)$&$(y+z)\epsilon_1+(x-z)\epsilon_2$&$-\epsilon_3+0$\\ \hline
$6$&$+$&$(y,x,-z)$&$(y-z)\epsilon_1+(x+z)\epsilon_2$&$-\epsilon_3-\epsilon_3$\\ \hline
$7$&$+$&$(z,x,y)$&$(y+z)\epsilon_1+(x-y)\epsilon_2$&$-\epsilon_3-\epsilon_2$\\ \hline
$8$&$-$&$(-z,x,y)$&$(y-z)\epsilon_1+(x-y)\epsilon_2$&$-\epsilon_3-\epsilon_1-\epsilon_2$\\ \hline
\end{tabular}
\end{center}
\caption{Projection of $w\lambda$ in $P^{+}_{M}$.}
\end{table}

In the picture below the location of the points $q(w_i\lambda)\in P_M^+$, indicated by the number $i$,
with the sign of $\det(w_i)$ attached, is drawn.
Observe that $q(w_1\lambda)=(k+m)\omega_1+l\omega_2\in P_K^+$ for all $\lambda=klm\in P_G^+$.

\begin{figure}[ht]
\begin{center}
\begin{tikzpicture}[scale=.6]


\pgfmathsetmacro\ax{2}
\pgfmathsetmacro\ay{0}
\pgfmathsetmacro\bx{2 * cos(120)}
\pgfmathsetmacro\by{2 * sin(120)}
\pgfmathsetmacro\lax{2*\ax/3 + \bx/3}
\pgfmathsetmacro\lay{2*\ay/3 + \by/3}
\pgfmathsetmacro\lbx{\ax/3 + 2*\bx/3}
\pgfmathsetmacro\lby{\ay/3 + 2*\by/3}


\begin{scope}

\clip (5/3,10/4) circle (5);

\draw[dashed] (0,0) -- (0,10);
\draw[dashed] (0,0)-- (5*\lax,5*\lay+5*\lby);
\draw[dashed] (0,0) -- (7*\lax,7*\lay);

\draw[thick,->] (0,0) -- (\ax,\ay);
\draw[thick,->] (0,0) -- (\bx,\by);
\draw[thick,->] (0,0) -- (-\ax,-\ay);
\draw[thick,->] (0,0) -- (-\bx,-\by);
\draw[thick,->] (0,0) -- (\ax+\bx,\ay+\by);
\draw[thick,->] (0,0) -- (-\ax-\bx,-\ay-\by);

\draw[thick,->] (0,0) -- (\lax,\lay);
\draw (\lax,\lay-1/8) node[right] {\(\eps_{2}\)};
\draw[thick,->] (0,0) -- (\lbx,\lby);
\draw (1/8+\lbx,\lby+1/8) node[above,left]{\(\eps_{1}\)};
\draw[thick,->] (0,0) -- (-\lax,-\lay);
\draw[thick,->] (0,0) -- (-\lbx,-\lby);
\draw[thick,->] (0,0) -- (\lax-\lbx,\lay-\lby);
\draw[thick,->] (0,0) -- (\lbx-\lax,\lby-\lay) node[left] {\(\eps_{3}\)};

\draw[thick, fill=lightgray] (4*\lbx,4*\lby) -- (4*\lbx+1/2*\lax,4*\lby-1/2*\lay) -- (2*\lax+7/2*\lbx,2*\lay+7/2*\lby) -- (2*\lax+8/2*\lbx,2*\lay+8/2*\lby) -- (1/2*\lax+11/2*\lbx,1/2*\lay+11/2*\lby) -- (0*\lax+11/2*\lbx,0*\lay+11/2*\lby) -- cycle;

\draw[fill] (0*\lax+8/2*\lbx,0*\lay+8/2*\lby) circle (2pt) node[left] {\(b+\)};
\draw[fill] (1/2*\lax+8/2*\lbx,-1/2*\lay+8/2*\lby) circle (2pt) node[below] {\(4-\)};
\draw[fill] (2*\lax,2*\lay+7/2*\lby) circle (2pt) node[below] {\(3+\)};

\draw[fill] (2*\lax+8/2*\lbx,2*\lay+8/2*\lby) circle (2pt) node[right] {\(2-\)};
\draw[fill] (1/2*\lax+11/2*\lbx,1/2*\lay+11/2*\lby) circle (2pt) node[above] {\(1+\)};
\draw[fill] (0*\lax+11/2*\lbx,0*\lay+11/2*\lby) circle (2pt) node[left] {\(a-\)};

\draw[thick] (0*\lax+8/2*\lbx,0*\lay+8/2*\lby) -- (2*\lax+8/2*\lbx,2*\lay+8/2*\lby);
\draw[thick] (1/2*\lax+8/2*\lbx,-1/2*\lay+8/2*\lby)--(1/2*\lax+11/2*\lbx,1/2*\lay+11/2*\lby);
\draw[thick] (2*\lax,2*\lay+7/2*\lby)--(0*\lax+11/2*\lbx,0*\lay+11/2*\lby);


\draw[thick, fill=lightgray] (4*\lax,4*\lay) -- (1/2*\lbx+7/2*\lax,1/2*\lby+7/2*\lay) -- (7/2*\lax+4/2*\lbx,7/2*\lay+4/2*\lby) -- (4*\lax+4/2*\lbx,4*\lay+4/2*\lby) -- (11/2*\lax+1/2*\lbx,11/2*\lay+1/2*\lby) -- (11/2*\lax+0/2*\lbx,11/2*\lay+0/2*\lby) -- cycle;

\draw[fill] (4*\lax+0/2*\lbx,4*\lay+0/2*\lby) circle (2pt) node[below] {\(d+\)};
\draw[fill] (7/2*\lax+1/2*\lbx,7/2*\lay+1/2*\lby) circle (2pt) node[left] {\(8-\)};
\draw[fill] (7/2*\lax+2*\lbx,7/2*\lay+2*\lby) circle (2pt) node[left] {\(7+\)};

\draw[fill] (4*\lax+2*\lbx,4*\lay+2*\lby) circle (2pt) node[above] {\(5-\)};
\draw[fill] (11/2*\lax+1/2*\lbx,11/2*\lay+1/2*\lby) circle (2pt) node[right] {\(6+\)};
\draw[fill] (11/2*\lax+0/2*\lbx,11/2*\lay+0/2*\lby) circle (2pt) node[below] {\(c-\)};

\draw[thick] (4*\lax+0/2*\lbx,4*\lay+0/2*\lby) -- (4*\lax+2*\lbx,4*\lay+2*\lby);
\draw[thick] (7/2*\lax+1/2*\lbx,7/2*\lay+1/2*\lby) -- (11/2*\lax+1/2*\lbx,11/2*\lay+1/2*\lby);
\draw[thick] (7/2*\lax+2*\lbx,7/2*\lay+2*\lby) -- (11/2*\lax+0/2*\lbx,11/2*\lay+0/2*\lby);

\end{scope}
\end{tikzpicture}

\end{center}
\caption{Projection of $W_{G}\lambda$ onto $P^{+}_{M}$.}\label{figure: Spin7G2 projection on M chamber}
\end{figure}
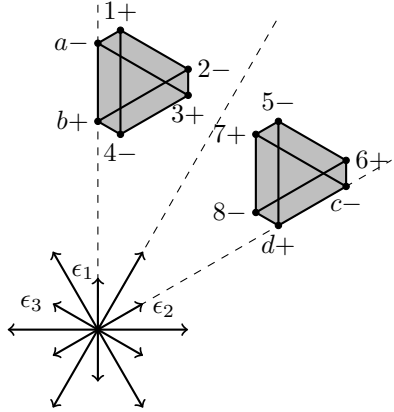Let us denote $a=(k+l+m)\epsilon_1$ and $b=(k+l)\epsilon_1$, and so these two points
together with the four points $q(w_i\lambda)$ for $i=1,2,3,4$ form the vertices of
a hexagon with three pairs of parallel sides. In the picture we have drawn all six
vertices in $P_K^+$, which happens if and only if $q(w_3\lambda)=k\epsilon_1 +(l+m)\epsilon_2\in P_K^+$,
or equivalently if $k\geq(l+m)$. But in general some of the $q(w_i\lambda)\in P_M^+$
for $i=2,3,4$ might lie outside $P_K^+$. Indeed $q(w_2\lambda)=(k+l)\epsilon_1+(l+m)\epsilon_2$
lies outside $P_K^+$ if $k<m$, and $q(w_4\lambda)=k\epsilon_1+l\epsilon_2$ lies outside $P_K^+$ if $k<l$.

For fixed $\lambda\in P_G^+$ the sum $m_{\lambda}(\mu)$ of
the following six partition functions as a function of $\mu\in P_K$
\[ \sum_1^4\det(w)p_A(q(w_i(\lambda+\rho_G)-\rho_G)-\mu)
-p_A(a-\epsilon_2-\mu)+p_A(b-\epsilon_1-\epsilon_2-\mu) \]
is just the familiar multiplicity function for the weight multiplicities of the root system $\mathrm{A}_2$.
It vanishes outside the hexagon with vertices $a$, $b$ and $q(w_i\lambda)$ for $i=1,2,3,4$.
On the outer shell hexagon it is equal to $1$,
and it steadily increases by $1$ for each inner shell hexagon,
untill the hexagon becomes a triangle,
and from that moment on it stabilizes on the inner triangle.
The two partition functions we have added corresponding to the points $a$ and $b$
are invariant as a function of $\mu$ for the action $\mu\mapsto s_2(\mu+\rho_K)-\rho_K$
of the simple reflection $s_2\in W_K$ with mirror $\R\omega_1$, because $s_2(A)=A$.
In order to obtain the final multiplicity function
\[ \mu\mapsto m_{\lambda}^{G,K}(\mu)=\sum_{v\in W_K}\det(v)m_{\lambda}(v(\mu+\rho_K)-\rho_K) \]
for the branching from $G$ to $K$ we have to antisymmetrize for the shifted by $\rho_K$ action of $W_K$.
Note that the two additional partition functions and their transforms under $W_K$
all cancel because of their symmetry and the antisymmetrization.
For $\mu\in P_K^+$ the only terms in the sum over $v\in W_K$ that have a nonzero contribution are
those for $v=e$ the identity element and $v=s_1$ the reflection with mirror $\R\omega_2$,
and we arrive at the following result.

\begin{theorem}\label{(B_3,G_2) multiplicity formula}
For $\lambda\in P_G^+$ and $\mu\in P_K^+$ the branching multiplicity from $G=\Spin(7)$ to $K=\mathrm{G}_2$ is given by
\begin{eqnarray}\label{eq:multB3G2}
m_{\lambda}^{G,K}(\mu)=m_{\lambda}(\mu)-m_{\lambda}(s_1\mu-\epsilon_3)
\end{eqnarray}
with $m_{\lambda}$ the weight multiplicty function of type $A_2$ as given by the
above alternating sum of the six partition functions.
\end{theorem}

Indeed, we have $s_1(\mu+\rho_K)-\rho_K=s_1\mu-\epsilon_3$.
As before, we denote $klm=k\varpi_1+l\varpi_2+m\varpi_3$ and $kl=k\omega_1+l\omega_2$ with $k,l,m\in\N$
for the highest weight of irreducible representations of $G$ and $K$ respectively.
For $\mu\in\N\omega_1$ the multiplicities $m_{\lambda}^{G,K}(\mu)$ are only governed by the first term on the right hand side of (\ref{eq:multB3G2}) with $v=e$
and so are equal to $1$ for $\mu=n0$ with $n=(k+l),\cdots,(k+l+m)$ and $0$ elsewhere.
Indeed, $\mu=n0$ has multiplicity one if and only if it is contained in the segment
from $b=(k+l)\epsilon_1$ to $a=(k+l+m)\epsilon_1$. This proves to following statement.

\begin{corollary}
The fundamental representation of $G$ with highest weight $\lambda=001$ is the spin representation
of dimension $8$ with $K$-types $\mu=10$ and $\mu=00$. It is the fundamental spherical representation
for the Gelfand pair $(G,K)$. The irreducible spherical representations of $G$
have highest weights $00m$ with $K$-spectrum the set $\{n0;0\leq n\leq m\}$.
\end{corollary}

\begin{corollary}\label{(B_3,G_2,F_1) multiplicity free}
For any irreducible representation of $G$ with highest weight $\lambda=klm$ all $K$-types
with highest weight $\mu\in F_1=\N\omega_1$ are multiplicity free, and the $K$-type with highest
weight $\mu=n0$ has multiplicity one if and only if $(k+l)\leq n\leq (k+l+m)$.
The domain of those $\lambda=klm$ for which the $K$-type $\mu=n0$ occurs has a well shape
$P^{+}_{G}(n0)=B(n0)+\N001$ with bottom
\[ B(n0)=\{klm\in P_G^+;k+l+m=n\} \]
given by a single linear relation.
\end{corollary}

\begin{proof}
The multiplicity freeness and the bounds for $n$ follow from Theorem \ref{(B_3,G_2) multiplicity formula} and in turn these inequalities $n\le k+l+m$ imply the formulae for $B(n0)$ and $P^{+}_{G}(n0)$.
\end{proof}

This ends our discussion that $(G,K,F_1=\N\omega_1)$ is a multiplicity free system.
In order to show that $(G,K,F_2=\N\omega_2)$ is also a multiplicity free triple
we shall carry out a similar analysis.

\begin{corollary}\label{(B_3,G_2,F_2) multiplicity free}
For an irreducible representation of $G$ with highest weight $\lambda=klm$ all $K$-types
with highest weight $\mu\in F_2=\N\omega_2$ are multiplicity free, and the K-type with highest
weight $\mu=0n$ has multiplicity one if and only if $\max(k,l)\leq n\leq \min(k+l,l+m)$.
The domain of those $\lambda=klm$ for which the $K$-type $\mu=0n$ occurs has a well shape
$P^{+}_{G}(0n)=B(0n)+\N001$ with bottom
\[ B(0n)=\{klm\in P_G^+;m\leq k\leq n,l+m=n\} \]
given by a single linear relation and inequalities.
\end{corollary}

\begin{proof}
Under the assumption of the first part of this proposition $klm\in P^{+}_{G}(0n)$ implies that $kl(m+1)\in P^{+}_{G}(0n)$, and the bottom $B(0n)$ of those $klm\in P^{+}_{G}(0n)$ for which $kl(m-1)\notin P^{+}_{G}(0n)$
contains $klm$ if and only if $n=l+m$ and $k\geq m$.
It remains to show the first part of the proposition.

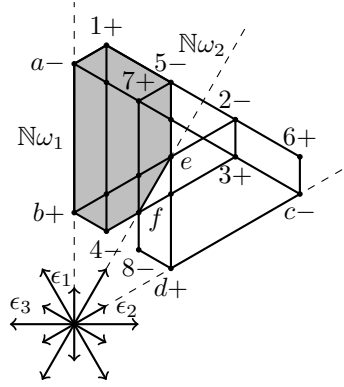
\begin{figure}[ht]
\begin{center}
\begin{tikzpicture}[scale=3/7]


\pgfmathsetmacro\ax{2}
\pgfmathsetmacro\ay{0}
\pgfmathsetmacro\bx{2 * cos(120)}
\pgfmathsetmacro\by{2 * sin(120)}
\pgfmathsetmacro\lax{2*\ax/3 + \bx/3}
\pgfmathsetmacro\lay{2*\ay/3 + \by/3}
\pgfmathsetmacro\lbx{\ax/3 + 2*\bx/3}
\pgfmathsetmacro\lby{\ay/3 + 2*\by/3}

\begin{scope}

\clip (5/3,13/4) circle (7);
\draw[dashed] (0,0) -- (0,10);
\draw[dashed] (0,0)-- (15*\lax,15*\lay+15*\lby);
\draw[dashed] (0,0) -- (15*\lax,15*\lay);

\draw[thick,->] (0,0) -- (\ax,\ay);
\draw[thick,->] (0,0) -- (\bx,\by);
\draw[thick,->] (0,0) -- (-\ax,-\ay);
\draw[thick,->] (0,0) -- (-\bx,-\by);
\draw[thick,->] (0,0) -- (\ax+\bx,\ay+\by);
\draw[thick,->] (0,0) -- (-\ax-\bx,-\ay-\by);

\draw(\lax,\lay-1/8) node[right]{\(\eps_{2}\)};
\draw[thick,->] (0,0) -- (\lax,\lay);
\draw[thick,->] (0,0) -- (\lbx,\lby);
\draw (2/8+\lbx,\lby+2/8) node[above,left]{\(\eps_{1}\)};
\draw[thick,->] (0,0) -- (-\lax,-\lay);
\draw[thick,->] (0,0) -- (-\lbx,-\lby);
\draw[thick,->] (0,0) -- (\lax-\lbx,\lay-\lby);
\draw[thick,->] (0,0) -- (\lbx-\lax,\lby-\lay) node[left] {\(\eps_3\)};

\draw[thick, fill=lightgray] (3*\lbx,3*\lby) -- (7*\lbx+0/2*\lax,7*\lby+0/2*\lay) -- (1*\lax+7*\lbx,1*\lay+7*\lby) -- (3*\lax+10/2*\lbx,3*\lay+10/2*\lby) -- (3*\lax+3*\lbx,3*\lay+3*\lby) -- (2*\lax+2*\lbx,2*\lay+2*\lby) -- (1*\lax+2*\lbx,1*\lay+2*\lby)-- cycle;

\draw[thick] (3*\lbx,3*\lby) -- (7*\lbx+0/2*\lax,7*\lby+0/2*\lay) -- (1*\lax+7*\lbx,1*\lay+7*\lby) -- (3*\lax+10/2*\lbx,3*\lay+10/2*\lby) -- (7*\lax+1*\lbx,7*\lay+1*\lby) -- (7*\lax+0*\lbx,7*\lay+0*\lby) -- (3*\lax+0*\lbx,3*\lay+0*\lby) -- (2*\lax+1*\lbx,2*\lay+1*\lby) --(2*\lax+2*\lbx,2*\lay+2*\lby) -- (5*\lax+2*\lbx,5*\lay+2*\lby) -- (7*\lax+0*\lbx,7*\lay+0*\lby);

\draw[thick] (3*\lax+0*\lbx,3*\lay+0*\lby) -- (3*\lax+3*\lbx,3*\lay+3*\lby);
\draw[thick] (0*\lax+7*\lbx,0*\lay+7*\lby) -- (5*\lax+2*\lbx,5*\lay+2*\lby);
\draw[thick] (0*\lax+3*\lbx,0*\lay+3*\lby) -- (5*\lax+3*\lbx,5*\lay+3*\lby);
\draw[thick] (2*\lax+2*\lbx,2*\lay+2*\lby) -- (2*\lax+5*\lbx,2*\lay+5*\lby) -- (3*\lax+5*\lbx,3*\lay+5*\lby);
\draw[thick] (5*\lax+2*\lbx,5*\lay+2*\lby) -- (5*\lax+3*\lbx,5*\lay+3*\lby);
\draw[thick] (1*\lax+2*\lbx,1*\lay+2*\lby) -- (1*\lax+7*\lbx,1*\lay+7*\lby);

\draw[fill] (0/2*\lax+6/2*\lbx,0/2*\lay+6/2*\lby) circle (2pt) node[left] {\(b+\)};
\draw[fill] (0/2*\lax+14/2*\lbx,0/2*\lay+14/2*\lby) circle (2pt) node[left] {\(a-\)};
\draw[fill] (1*\lax+7*\lbx,1*\lay+7*\lby) circle (2pt) node[above] {\(1+\)};
\draw[fill] (3*\lax+10/2*\lbx,3*\lay+10/2*\lby) circle (2pt) node[above] {\(5-\)};
\draw[fill] (3*\lax+3*\lbx,3*\lay+3*\lby) circle (2pt);
\draw (3*\lax+3*\lbx,3*\lay+3*\lby-1/5) node [right] {\(e\)};
\draw (2*\lax+2*\lbx,2*\lay+2*\lby-1/5) node[right] {\(f\)};
\draw[fill] (2*\lax+2*\lbx,2*\lay+2*\lby) circle (2pt);
\draw[fill] (1*\lax+2*\lbx,1*\lay+2*\lby) circle (2pt) node[below] {\(4-\)};

\draw[fill] (1*\lax+3*\lbx,1*\lay+3*\lby) circle (2pt);
\draw[fill] (1*\lax+6*\lbx,1*\lay+6*\lby) circle (2pt);
\draw[fill] (2*\lax+5*\lbx,2*\lay+5*\lby) circle (2pt) node[above] {\(7+\)};
\draw[fill] (3*\lax+4*\lbx,3*\lay+4*\lby) circle (2pt);
\draw[fill] (2*\lax+3*\lbx,2*\lay+3*\lby) circle (2pt);

\draw[fill] (3*\lax+0*\lbx,3*\lay+0*\lby) circle (2pt)  node[below] {\(d+\)};
\draw[fill] (14/2*\lax+0/2*\lbx,14/2*\lay+0/2*\lby) circle (2pt) node[below] {\(c-\)};

\draw[fill] (14/2*\lax+1*\lbx,14/2*\lay+1*\lby) circle (2pt) node[above] {\(6+\)};
\draw[fill] (5*\lax+3*\lbx,5*\lay+3*\lby) circle (2pt) node[above] {\(2-\)};
\draw[fill] (2*\lax+1*\lbx,2*\lay+1*\lby) circle (2pt) node[below] {\(8-\)};
\draw[fill] (3*\lax+2*\lbx,3*\lay+2*\lby) circle (2pt);
\draw[fill] (5*\lax+2*\lbx,5*\lay+2*\lby) circle (2pt) node[below] {\(3+\)};

\draw (0,5*\lby) node[left] {\(\mathbb{N}\omega_{1}\)};
\draw (5*\lax+5*\lbx,5*\lay+5*\lby) node[left] {\(\mathbb{N}\omega_{2}\)};

\end{scope}
\end{tikzpicture}
\end{center}
\caption{Support of the multiplicity function $\mu\mapsto m^{G,K}_{\lambda}(\mu)$.}\label{figure: Spin7G2 support}
\end{figure}
In order to determine the $K$-spectrum associated to the highest weight $\lambda=klm\in\N^3$ for $G$ observe that
\[ q(w_3\lambda)=k\epsilon_1+(l+m)\epsilon_2 \]
and so the $K$-spectrum on $\N\omega_2$ is empty for $k>(l+m)$, while
for $k=(l+m)$ the $K$-spectrum has a unique point $k\omega_2$ on $\N\omega_2$.
If $k<(l+m)$ the point $q(w_3\lambda)$ moves out of the dominant cone $P_K^+$ into $P_M^+-P_K^+$,
and the support of the function $P_K^+\ni\mu\mapsto m_{\lambda}^{G,K}(\mu)$ consists of
(the integral points of) a heptagon with an additional side on $\N\omega_2$ from $e$ to $f$
as in the picture above. On the outer shell heptagon the multiplicity is one,
and the multiplicities increase by one for each inner shell heptagon,
untill the heptagon becomes a triangle or quadrangle, and it stabilizes.
This follows from Theorem{\;\ref{(B_3,G_2) multiplicity formula}} in a straightforward way.

Depending on whether the vertex
\[ q(w_2\lambda)=(k+l)\epsilon_1+(l+m)\epsilon_2 \]
lies in $P_K^+$ (for $k\geq m$) or in $P_M^+-P_K^+$ (for $k<m$)
we get $e=(l+m)\omega_2$ or $e=(k+l)\omega_2$ respectively.
Hence we find
\[ e=\min(k+l,l+m)\omega_2\;,\;f=\max(k,l)\omega_2 \]
by a similar consideration for
\[ q(w_4\lambda)=k\epsilon_1+l\epsilon_2 \]
as before ($f=k$ for $k\geq l$ and $f=l$ for $k<l$).
This finishes the proof of Corollary{\;\ref{(B_3,G_2,F_2) multiplicity free}}.
\end{proof}

Our choice of positive roots for $G=\mathrm{B}_3$ and $K=\mathrm{G}_2$ was made in such a way that
the dominant cone $P_K^+$ for $K$ was contained in the dominant cone $P_G^+$ for $G$.
In turn this implies that the set
\[ A=q(R_G^+)-R_K^+=\{\epsilon_1,\epsilon_1,\epsilon_3\}\]
lies in an open half plane, which was required for the application
of the branching rule of Lemma{\;\ref{branching rule lemma}}.

However, we now switch to a different positive system in $R_G$, or rather
we keep $R_G^+$ fixed as before, but take the Lie algebra $\mathfrak{k}$ of $\mathrm{G}_2$ to be
perpendicular to the spherical direction $\varpi_3=(e_1+e_2+e_3)/2$ instead.
Under this assumption the positive roots $R_M^+$ form a parabolic subsystem in $R_G^+$,
and so the simple roots $\{\alpha_1=e_1-e_2,\alpha_2=e_2-e_3\}$ of $R_M^+$ are also simple roots in $R_G^+$.

\begin{figure}[ht]
\begin{center}
\begin{tikzpicture}[scale=.7]


\pgfmathsetmacro\ax{2}
\pgfmathsetmacro\ay{0}
\pgfmathsetmacro\bx{2 * cos(120)}
\pgfmathsetmacro\by{2 * sin(120)}
\pgfmathsetmacro\lax{2*\ax/3 + \bx/3}
\pgfmathsetmacro\lay{2*\ay/3 + \by/3}
\pgfmathsetmacro\lbx{\ax/3 + 2*\bx/3}
\pgfmathsetmacro\lby{\ay/3 + 2*\by/3}


\begin{scope}

\clip (-2,-2) rectangle (3+4*\ax,8);

\draw[thick, fill=lightgray] (0,0) -- (4*\lax,4*\lay) -- (0,4*\lby) -- cycle;

\draw[dashed] (0,0) -- (0,10);
\draw[dashed] (0,0)-- (8*\lax,8*\lay+8*\lby);
\draw[dashed] (0,0) -- (5*\lax,5*\lay);

\draw[thick,->] (0,0) -- (\ax,\ay) node[right] {\(\alpha_{2}\)};
\draw[thick,->] (0,0) -- (\bx,\by) node[above] {\(\alpha_{1}\)};
\draw[thick,->] (0,0) -- (-\ax,-\ay);
\draw[thick,->] (0,0) -- (-\bx,-\by);
\draw[thick,->] (0,0) -- (\ax+\bx,\ay+\by);
\draw[thick,->] (0,0) -- (-\ax-\bx,-\ay-\by);

\draw(\lax,\lay-1/8) node[right]{\(\eps_{2}\)};
\draw[thick,->] (0,0) -- (\lax,\lay);
\draw[thick,->] (0,0) -- (\lbx,\lby);
\draw (\lbx,\lby) node[left]{\(\eps_{1}\)};
\draw[thick,->] (0,0) -- (-\lax,-\lay);
\draw[thick,->] (0,0) -- (-\lbx,-\lby);
\draw[thick,->] (0,0) -- (\lax-\lbx,\lay-\lby);
\draw[thick,->] (0,0) -- (\lbx-\lax,\lby-\lay) node[left] {\(\eps_3\)};

\fill (0,4*\lby) circle (2pt) node[left] {\(n\eps_{1}\)};
\fill (4*\lax,4*\lay) circle (2pt);
\draw (4*\lax,4*\lay-1/4) node[below] {\(n\eps_{1}\)};

\draw[thick, fill=lightgray] (4*\lax+4*\lbx+3*\ax,4*\lay+4*\lby) -- (4*\lax+3*\ax,4*\lay) -- (0+3*\ax,4*\lby) -- cycle;

\draw[dashed] (0+3*\ax,0) -- (0+3*\ax,10);
\draw[dashed] (0+3*\ax,0)-- (15*\lax+3*\ax,15*\lay+15*\lby);
\draw[dashed] (0+3*\ax,0) -- (15*\lax+3*\ax,15*\lay);

\draw[thick,->] (0+3*\ax,0) -- (\ax+3*\ax,\ay) node[right] {\(\alpha_{2}\)};
\draw[thick,->] (0+3*\ax,0) -- (\bx+3*\ax,\by) node[above] {\(\alpha_{1}\)};
\draw[thick,->] (0+3*\ax,0) -- (-\ax+3*\ax,-\ay);
\draw[thick,->] (0+3*\ax,0) -- (-\bx+3*\ax,-\by);
\draw[thick,->] (0+3*\ax,0) -- (\ax+\bx+3*\ax,\ay+\by);
\draw[thick,->] (0+3*\ax,0) -- (-\ax-\bx+3*\ax,-\ay-\by);

\draw(\lax+3*\ax,\lay-1/8) node[right]{\(\eps_{2}\)};
\draw[thick,->] (0+3*\ax,0) -- (\lax+3*\ax,\lay);
\draw[thick,->] (0+3*\ax,0) -- (\lbx+3*\ax,\lby);
\draw (\lbx+3*\ax,\lby) node[left]{\(\eps_{1}\)};
\draw[thick,->] (0+3*\ax,0) -- (-\lax+3*\ax,-\lay);
\draw[thick,->] (0+3*\ax,0) -- (-\lbx+3*\ax,-\lby);
\draw[thick,->] (0+3*\ax,0) -- (\lax-\lbx+3*\ax,\lay-\lby);
\draw[thick,->] (0+3*\ax,0) -- (\lbx-\lax+3*\ax,\lby-\lay) node[left] {\(\eps_3\)};

\fill (03*\ax,4*\lby) circle (2pt) node[left] {\(n\eps_{1}\)};
\fill (4*\lax+3*\ax,4*\lay) circle (2pt);
\fill (4*\lax+4*\lbx+3*\ax,4*\lay+4*\lby) circle (2pt);
\draw (4*\lax+3*\ax,4*\lay-1/4) node[below] {\(n\eps_{1}\)};
\draw (4*\lax+4*\lbx+3*\ax-1/4,4*\lay+4*\lby) node[left] {\(n(\eps_{1}+\eps_{2})\)};

\end{scope}
\end{tikzpicture}
\end{center}
\caption{Projections of the bottoms $B_{n0}$ and $B_{0n}$.}\label{figure: Spin7G2 projections of bottoms}
\end{figure}
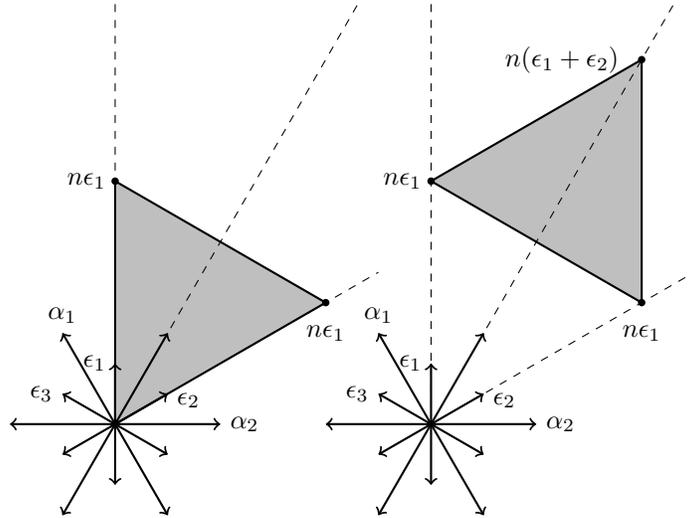

Let $p:P_G\rightarrow P_M=P_K$ be the orthogonal projection along the spherical direction.
By abuse of notation we denote (with $p(\varpi_3)=0$)
\[ \epsilon_1=p(\varpi_1)=(2,-1,-1)/3\;,\;\epsilon_2=p(\varpi_2)=(1,1,-2)/3 \]
for the fundamental weights for $P_M^+=p(P_G^+)$.
It is now easy to check that this projection
\[ p:B(n0)\rightarrow p(B(n0))\;,\;p:B(0n)\rightarrow p(B(0n)) \]
is a bijection from the bottom onto its image in $P_M^+$.
In the Figure \ref{figure: Spin7G2 projections of bottoms} we have drawn the projections
\[ p(B(n0))=\{k\epsilon_1+l\epsilon_2;k+l\leq n\}\;,\;p(B(0n))=\{k\epsilon_1+l\epsilon_2;k,l\leq n,k+l\geq n\} \]
on the left and the right side respectively.

Let us now prove the remaining case of Proposition \ref{prop: asymptotics}. Consider $\lambda=klm\in P^{+}_{G}$.
We take $x=k+l+m/2,y=l+m/2,z=m/2$ with $m$ relatively large. The projections of the elements $w\lambda$ that land in $P^{+}_{M}$ are given in Table \ref{Table: projections on M chamber}.

\begin{table}[ht]

\begin{center}
\begin{tabular}{|l|l|l|l|l|} \hline
$i$&$w_i\lambda$&$q(w_i\lambda)$&$q(w_i\lambda)$\\ \hline
$1$&$(x,y,z)$&$(x+z)\epsilon_1+(y-z)\epsilon_2$&$(k+l+m)\epsilon_1+l\epsilon_2$\\ \hline
$2$&$(x,y,-z)$&$(x-z)\epsilon_1+(y+z)\epsilon_2$&$(k+l)\epsilon_1+(l+m)\epsilon_2$\\ \hline
$3$&$(x,z,-y)$&$(x-y)\epsilon_1+(y+z)\epsilon_2$&$k\epsilon_1+(l+m)\epsilon_2$\\ \hline
$4$&$(x,-z,-y)$&$(x-y)\epsilon_1+(y-z)\epsilon_2$&$k\epsilon_1+l\epsilon_2$\\ \hline
$5$&$(y,x,z)$&$(y+z)\epsilon_1+(x-z)\epsilon_2$&$(l+m)\epsilon_1+(k+l)\epsilon_2$\\ \hline
$6$&$(y,x,-z)$&$(y-z)\epsilon_1+(x+z)\epsilon_2$&$l\epsilon_1+(k+l+m)\epsilon_2$\\ \hline
$7$&$(z,x,y)$&$(y+z)\epsilon_1+(x-y)\epsilon_2$&$(l+m)\epsilon_1+k\epsilon_2$\\ \hline
$8$&$(-z,x,y)$&$(y-z)\epsilon_1+(x-y)\epsilon_2$&$l\epsilon_1+k\epsilon_2$\\ \hline
\end{tabular}
\end{center}
\caption{Projections of $w\lambda$ in $P^{+}_{M}$.}\label{Table: projection of lambda}\label{Table: projections on M chamber}
\end{table}

As $m$ gets large the points $q(w_i\lambda)$ run to infinity except for $i=4$ and $i=8$. 
This means that we should take for $\nu=p(\lambda)=q(w_4\lambda)$ if we pick $i=4$. 
The multiplicity behavior $m^{G,K}_{\lambda}(\mu)$ in Picture \ref{figure: Spin7G2 support} for $m\rightarrow\infty$ 
goes as a function of $\mu\in P_K^+=\N\omega_1+\N\omega_2$ to the function that gives the multiplicity of $\mu$ induced 
representation $\mathrm{Ind}_{M}^{K}(V^{M}_{\nu})$ from $M=\SU(3)$ to $K=\G_2$, 
and therefore by Frobenius reciprocity equals $m^{K,M}_{\mu}(\nu)$. This shows that $\lim_{m\to\infty}m^{G,K}_{\lambda}(\mu)=m^{K,M}_{\mu}(\nu)$. 

The weights of the fundamental spherical representation with highest weight $\lambda_{\sph}=\varpi_{3}$ are $\frac{1}{2}(\pm\eps_{1}\pm\eps_{2}\pm\eps_{3})$. Expressed in terms of fundamental weights these become
$$001,(-1)01,1(-1)1, 01(-1)$$
and their negatives. It follows from Corollaries \ref{(B_3,G_2,F_1) multiplicity free} and \ref{(B_3,G_2,F_2) multiplicity free} that Theorem \ref{degree inequality theorem} holds true for this case.


\section{The pair $(G,K)=(\USp(2n),\USp(2n-2)\times\USp(2))$}\label{symplectic}

Let $G=\USp(2n)$ and $K=\USp(2n-2)\times\USp(2)$ with $n\ge3$. The weight lattices of $G$ and $K$ are equal, $P=\bbZ^{n}$, and we denote by $\eps_{i}$ the $i$-th basis vector. The set of dominant weights for $G$ is $P^{+}_{G}=\{(a_{1},\ldots,a_{n})\in P:a_{1}\ge\ldots\ge a_{n}\ge0\}$. The set of dominant weights for $K$ is  $P^{+}_{K}=\{(b_{1},\ldots,b_{n})\in\ P:b_{1}\ge\ldots\ge b_{n-1}\ge0,b_{n}\ge0\}$. The branching rule from $G$ to $K$ is due to Lepowsky \cite{Lepowsky}, \cite[Thm.~9.50]{Knapp}.

\begin{theorem}[Lepowsky]\label{theorem: lepowsky1}
Let $\lambda=(a_{1},\ldots,a_{n})\in P^{+}_{G}$ and $\mu=(b_{1},\ldots,b_{n})\in P^{+}_{K}$. Define $A_{1}=a_{1}-\max(a_{2},b_{1})$, $A_{k}=\min(a_{k},b_{k-1})-\max(a_{k+1},b_{k})$ for $2\le k\le n-1$ and $A_{n}=\min(a_{n},b_{n-1})$.
The multiplicity $m^{G,K}_{\lambda}(\mu)=0$ unless all $A_{i}\ge0$ and $b_{n}+\sum_{i=1}^{n}A_{i}\in2\bbZ$. In this case the multiplicity is given by
\begin{multline}\label{eq: lepowsky}
m^{G,K}_{\lambda}(\mu)=p_{\Sigma}(A_{1}\eps_{1}+A_{2}\eps_{2}+\cdots+(A_{n}-b_{n})\eps_{n})-\\
p_{\Sigma}(A_{1}\eps_{1}+A_{2}\eps_{2}+\cdots+(A_{n}+b_{n}+2)\eps_{n})
\end{multline}
where $p_{\Sigma}$ is the multiplicity function for the set $\Sigma=\{\eps_{i}\pm\eps_{n}:1\le i\le n-1\}$.
\end{theorem}

\begin{theorem}
Let $\mu=x\omega_{i}+y\omega_{j}\in P^{+}_{K}$ with $i<j$ and write $\mu=(b_{1},\ldots,b_{n})$. Let $\lambda=(a_{1},\ldots,a_{n})\in P^{+}_{G}$. Let $A_{1},\ldots, A_{n}$ be defined as in Theorem \ref{theorem: lepowsky1}. Then $m^{G,K}_{\lambda}(\mu)\le1$ with equality precisely when (1) $A_{k}\ge0$ for $k=1,\ldots,n-1$, (2) $b_{n}+\sum_{k=1}^{n}A_{k}\in2\bbZ$ and (3) $\max(A_{k},b_{n})\le\frac{1}{2}(b_{n}+\sum_{k=1}^{n}A_{k})$.
\end{theorem}

\begin{proof}
Suppose that $m^{G,K}_{\lambda}(\mu)\ge1$. Then (1) and (2) follow from Theorem \ref{theorem: lepowsky1}. In fact, $A_{k}=0$ unless $k\in\{1,i+1,j+1\}\cap[1,n]$, because of the hypothesis on $\mu$. We evaluate (\ref{eq: lepowsky}) below, showing that $m^{G,K}_{\lambda}(\mu)\le1$ with equality precisely when (3) holds.

We distinguish 4 cases: (i) $j<n-1$, (ii) $j=n-1$, (iii) $j=n, i=n-1$, (iv) $j=n, i<n-1$. In all cases we reduce to $n=4$ and we find the following expressions for $m_{\lambda}^{G,K}(\mu)$: 
\begin{itemize}
\item[(i)] $p_{\Sigma}(A_{1}\eps_{1}+A_{2}\eps_2+A_{3}\eps_{3})-p_{\Sigma}(A_{1}\eps_{1}+A_{2}\eps_2+A_{3}\eps_{3}+2\eps_{4})$,
\item[(ii)] $p_{\Sigma}(A_{1}\eps_{1}+A_{2}\eps_2+A_{4}\eps_{4})-p_{\Sigma}(A_{1}\eps_{1}+A_{2}\eps_2+(A_{4}+2)\eps_{4})$,
\item[(iii)] $p_{\Sigma}(A_{1}\eps_{1}+(A_{4}-b_{4})\eps_{4})-p_{\Sigma}(A_{1}\eps_{1}+(A_{4}+b_{4}+2)\eps_{4})$,
\item[(iv)] $p_{\Sigma}(A_{1}\eps_{1}+A_{2}\eps_2-b_{4}\eps_{4})-p_{\Sigma}(A_{1}\eps_{1}+A_{2}\eps_2+(b_{4}+2)\eps_{4})$.
\end{itemize}
The cases (ii) and (iv) reduce to (i) using elementary manipulations of partition functions, see \cite[p.~588]{Knapp}. Case (iii) can also be reduced to (i) but this is not necessary as we can handle this case directly. We have $p_{\Sigma}(A_{1}\eps_{1}+(A_{4}-b_{4})\eps_{4})\le1$ with equality if and only if $A_{1}+A_{4}-b_{4}\in2\bbN$ and $A_{1}-A_{4}+b_{4}\in2\bbN$. Similarly $p_{\Xi}(A_{1}\eps_{1}+(A_{4}+b_{4}+2)\eps_{4})\le1$ with equality if and only if $A_{1}+A_{4}+b_{4}+2\in2\bbN$ and $A_{1}-A_{4}-b_{4}-2\in2\bbN$. This implies the assertion in case (iii).

In case (i) we have
\begin{eqnarray*}
\sum_{k=1}^{3}A_{k}\eps_{k}=\sum_{k=1}^{3}B_{k}(\eps_{k}+\eps_{4})+\sum_{k=1}^{3}(A_{k}-B_{k})(\eps_{k}-\eps_{4})
\end{eqnarray*}
if and only if $\sum_{i=1}^{3}B_{k}=A$. It follows that
$$p_{\Sigma}(A_{1}\eps_{1}+A_{2}\eps_{2}+A_{3}\eps_{3})=
\#\{(B_{1},B_{2},B_{3})\in\bbN^{3}:\sum_{k=1}^{3}B_{k}=A\mbox{ and }B_{k}\le A_{k}\}$$
and similarly
\begin{multline}p_{\Sigma}(A_{1}\eps_{1}+A_{2}\eps_{2}+A_{3}\eps_{3}+2\eps_{4})=\\
\#\{(B_{1},B_{2},B_{3})\in\bbN^{3}:\sum_{k=1}^{3}B_{k}=A+1\mbox{ and }B_{k}\le A_{k}\}.
\end{multline}
Assume that $A_{1}\ge A_{2}\ge A_{3}$. We distinguish two possibilities: (1) $A_{1}\le A$ and (2) $A_{1}>A$. In case (1) we have
\[p_{\Sigma}(\sum_{i=1}^{3}A_{i}\eps_{i})=\#\{\mbox{lattice points in hexagon indicated in Figure \ref{figure: counting points}}\}\]
which is given by
\[p_{\Sigma}(\sum_{i=1}^{3}A_{i}\eps_{i})=(A+1)(A+2)/2-\sum_{i=1}^{3}(A-A_{i})(A-A_{i}+1)/2.\]
Similarly
\[p_{\Sigma}(\sum_{i=1}^{3}A_{i}\eps_{i}+2\eps_{4})=(A+2)(A+3)/2-\sum_{i=1}^{3}(A+1-A_{i})(A-A_{i}+2)/2\]
and the difference is one, as was to be shown.

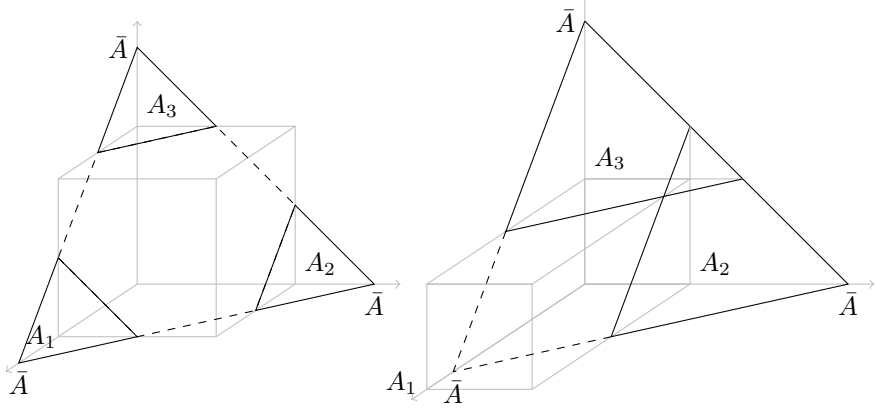
\begin{figure}[ht]
\begin{center}
\begin{tikzpicture}[scale=.35]
\begin{scope}

\clip (-5,-5) rectangle (28,11);


\draw[thin,lightgray,->] (0,0) -- (0,10);
\draw[thin,lightgray] (0,6) -- (6,6) -- (3,4) -- (3,-2);
\draw[thin,lightgray] (6,6) -- (6,0);
\draw[thin,lightgray] (6,0) -- (3,-2);
\draw[thin,lightgray] (3,-2) -- (-3,-2);
\draw[thin,lightgray] (-3,-2) -- (-3,4) --(3,4);
\draw[thin,lightgray] (-3,4) -- (0,6);
\draw[thin,lightgray,->] (0,0) -- (-5,-10/3);
\draw[thin,lightgray,->] (0,0) -- (10,0);


\draw[thin] (0,9) node[left]{$\bar{A}$} -- (3,6) -- (-3/2,5) -- cycle;
\draw[thin] (9,0) node[below]{$\bar{A}$} -- (9/2,-1) -- (6,3) -- cycle;
\draw[thin] (-9/2,-3) node[below]{$\bar{A}$} -- (-3,1) -- (0,-2) -- cycle;
\draw[dashed] (-3,1) -- (0,-2) -- (9/2,-1) -- (6,3) -- (3,6) -- (-3/2,5) -- cycle;

\draw[thin] (6,0) node[above right]{$A_2$};
\draw[thin] (0,6) node[above right]{$A_3$};
\draw[thin] (-2.7,-2) node[left]{$A_1$};

\draw[thin,lightgray,->] (20-3,0) -- (31-3,0);
\draw[thin,lightgray,->] (20-3,0) -- (20-3,11);
\draw[thin,lightgray,->] (20-3,0) -- (67/5-3,-22/5);
\draw[thin,lightgray] (14-3,-4) -- (18-3,-4) -- (18-3,0) -- (14-3,0) -- cycle; 
\draw[thin,lightgray] (24-3,0) -- (20-3,0) -- (20-3,4) -- (24-3,4) -- cycle; 
\draw[thin,lightgray] (14-3,-4) -- (20-3,0) -- (24-3,0) -- (18-3,-4);
\draw[thin,lightgray] (14-3,0) -- (20-3,4) -- (24-3,4) -- (18-3,0) -- cycle;

\draw[thin] (14-3,-4+1/4) node[left]{$A_1$};
\draw[thin] (24-3,0) node[above right]{$A_2$};
\draw[thin] (20-3,4) node[above right]{$A_3$};

\draw[thin] (20-3,10) node[left]{$\bar{A}$} -- (30-3,0) node[below]{$\bar{A}$}  -- (21-3,-2);
\draw[dashed] (21-3,-2) -- (15-3,-10/3) node[below]{$\bar{A}$};
\draw[thin] (21-3,-2) -- (24-3,6); 
\draw[thin, lightgray] (24-3,6) -- (24-3,4) -- (26-3,4);
\draw[thin] (23-3,10/3) -- (26-3,4);
\draw[thin] (23-3,10/3) -- (17-3,2) -- (20-3,10);
\draw[dashed] (17-3,2) -- (15-3,-10/3);


\end{scope}

\end{tikzpicture}
\end{center}
\caption{Counting integral points.}\label{figure: counting points}
\end{figure}
In case (2) where $A_{1}>A$ we have
\[p_{\Sigma}(\sum_{i=1}^{3}A_{i}\eps_{i})=\#\{\mbox{lattice points in parallelogram in Figure \ref{figure: counting points}}\}\]
which is given by $A_{2}A_{3}$. Similarly $p_{\Sigma}(\sum_{i=1}^{3}A_{i}\eps_{i}+2\eps_{4})=A_{2}A_{3}$ and hence the difference is zero. 
\end{proof}

The bottom $B(\mu)$ of the $\mu$-well $P^{+}_{G}(\mu)$ is parametrized by $P^{+}_{M}(\mu)$, where $M\cong\USp(2)\times\USp(2n-4)\times\USp(2)$. In \cite{Baldoni Silva} the branching rules for $K$ to $M$ are described. The dominant integral weights for $M$ are parametrized by $P^{+}_{M}=\{(c_{1},c_{2},\ldots,c_{n-1},c_{1}): 2c_{1}\in\bbN, c_{2}\ge\cdots\ge c_{n-1}\}\subset P$. The map $p:P^{+}_{G}\to P^{+}_{M}$ from Proposition \ref{projection from induced to restricted spectrum} is given as follows. Write $\lambda=(a_{1},\ldots,a_{n})\in P^{+}_{G}$ as
\begin{eqnarray}
\lambda=(\lambda-\frac{a_{1}+a_{2}}{2}\lambda_{\sph})+\frac{a_{1}+a_{2}}{2}\lambda_{\sph},
\end{eqnarray}
with $\lambda_{\sph}=\varpi_{2}=\eps_{1}+\eps_{2}$. Then $p(\lambda)=(\frac{1}{2}(a_{1}+a_{2}),a_{3},\ldots,a_{n-1},\frac{1}{2}(a_{1}+a_{2}))\in P^{+}_{M}$. The map $q:P\to P:\lambda\mapsto\lambda-(a_{1}+a_{2})\lambda_{\sph}/2$ projects onto the orthocomplement of $\lambda_{\sph}$ and the maps $p$ and $q$ differ by a Weyl group element in $W_{G}$. To determine the bottom $B(\mu)$ we have to find for each $\lambda\in P^{+}_{G}(\mu)$ the minimal $d\in\frac{1}{2}\bbN$ for which $q(\lambda)+d\lambda_{\sph}\in P^{+}_{G}(\mu)$. We distinguish two cases for the $K$-type $\mu=x\omega_{i}+y\omega_{j}=(b_{1},\ldots,b_{n})$, $i<j$: (1) $i=1$, (2) $i>1$. Assume (1). Then the relevant inequalities are $A_{1}\ge0,A_{2}\ge0$ and $A_{1}+A_{2}\ge B$, with $B$ equal to $A_{j+1}$ or $y$, depending on $j<n$ or $j=n$ respectively. Plugging in $\lambda=q(\lambda)+d\lambda_{\sph}$ and minimizing for $d$ yields
$$d=\max(b_{1}-c_{1},b_{2}+c_{1},\frac{1}{2}(b_{1}+B+\max(a_{3},b_{2}))),$$
where $c_{1}=(a_{1}-a_{2})/2$. The branching rules for $K$ to $M$ specialized to the specific choice of $\mu$ implies that $d=\frac{1}{2}(b_{1}+B+\max(a_{3},b_{2}))$ (see \cite{van Pruijssen Roman}). Assume (2). The relevant inequality is $A_{1}\ge0$. Since $i>1$ we have $b_{1}=b_{2}$ so $A_{1}=a_{1}-a_{2}$, which is invariant for adding multiples of $\lambda_{\sph}$. We plug in $q(\lambda)+d\lambda_{\sph}$ and write $c_{1}=(a_{1}-a_{2})/2$ . Minimizing $d$ so that $A_{k}\ge0$ yields $d=c_{1}+b_{1}$.

The weights of the fundamental spherical representation of highest weight $\lambda_{\sph}=\eps_{1}+\eps_{2}$ are $\{\pm\eps_{i}\pm\eps_{j}:i<j\}\cup\{0\}$. One easily checks that Theorem \ref{degree inequality theorem} holds true for this case.


\section{The pair $(G,K)=(\F_{4},\Spin(9))$}\label{F4}

In this section we take $G$ of type $F_{4}$ and $K=\Spin(9)$ the subgroup of type $B_{4}$. Let $H\subset K\subset G$ be the standard maximal torus and let $\lag,\lak,\lah$ denote the corresponding Lie algebras. We fix the set of positive roots of the root systems $\Delta(\lag,\lah)$ and $\Delta(\lak,\lah)$,
$$R^{+}_{K}=\{\eps_{i}\pm\eps_{j}|1\le i< j\le4\}\cup\{\eps_{i}|1\le i\le4\},$$
$$R^{+}_{G}=R^{+}_{K}\cup\left\{\frac{1}{2}(\eps_{1}\pm\eps_{2}\pm\eps_{2}\pm\eps_{2})\right\}.$$
The corresponding systems of simple roots are
$$\Pi_{G}=\{\alpha_{1}=\frac{1}{2}(\eps_{1}-\eps_{2}-\eps_{3}-\eps_{4}), \alpha_{2}=\eps_{4},\alpha_{3}=\eps_{3}-\eps_{4},\alpha_{4}=\eps_{2}-\eps_{3}\},$$
$$\Pi_{K}=\{\beta_{1}=\eps_{1}-\eps_{2},\beta_{2}=\eps_{2}-\eps_{3},\beta_{3}=\eps_{3}-\eps_{4},\beta_{4}=\eps_{4}\},$$
see also the Dynkin diagram in Figure \ref{figure: dynkin}.

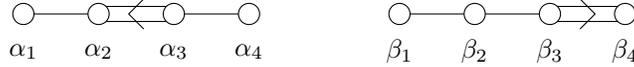
\begin{figure}[ht]
\begin{center}
\begin{tikzpicture}[scale=1]
\begin{scope}

\clip (0,-.30) rectangle (10.1,1);



\draw[black] (1,0.5) -- (2,0.5);
\draw[black] (3,0.5) -- (4,0.5);
\draw[black] (6,0.5) -- (8,0.5);

\draw[black] (2,0.4) -- (3,0.4);
\draw[black] (2,0.6) -- (3,0.6);

\draw[black] (8,0.4) -- (9,0.4);
\draw[black] (8,0.6) -- (9,0.6);

\draw[fill=white] (1,0.5) circle (4pt);
\draw[fill=white] (2,0.5) circle (4pt);
\draw[fill=white] (3,0.5) circle (4pt);
\draw[fill=white] (4,0.5) circle (4pt);

\draw[fill=white] (6,0.5) circle (4pt);
\draw[fill=white] (7,0.5) circle (4pt);
\draw[fill=white] (8,0.5) circle (4pt);
\draw[fill=white] (9,0.5) circle (4pt);

\draw (1,0) node {\(\alpha_{1}\)};
\draw (2,0) node {\(\alpha_{2}\)};
\draw (3,0) node {\(\alpha_{3}\)};
\draw (4,0) node {\(\alpha_{4}\)};

\draw (6,0) node {\(\beta_{1}\)};
\draw (7,0) node {\(\beta_{2}\)};
\draw (8,0) node {\(\beta_{3}\)};
\draw (9,0) node {\(\beta_{4}\)};

\draw[black] (2.6 ,0.7) -- (2.4 ,0.5) --(2.6 ,0.3);

\draw[black] (8.4 ,0.7) -- (8.6 ,0.5) --(8.4 ,0.3);

\end{scope}

\end{tikzpicture}
\end{center}
\caption{The Dynkin diagrams of $\F_{4}$ and $B_{4}$.}\label{figure: dynkin}
\end{figure}
\noindent The fundamental weights corresponding to $\Pi_{G}$ are given by
$$\varpi_{1}=\eps_{1},\varpi_{2}=\frac{1}{2}(3\eps_{1} + \eps_{2} + \eps_{3} + \eps_{4}), \varpi_{3}= 2 \eps_{1} + \eps_{2} + \eps_{3}, \varpi_{4}=\eps_{1} + \eps_{2}$$
and those corresponding to $\Pi_{K}$ by
$$\omega_{1}=\eps_{1},\omega_{2}=\eps_{1}+\eps_{2},\omega_{3}=\eps_{1}+\eps_{2}+\eps_{3},\omega_{4}=\frac{1}{2}(\eps_{1}+\eps_{2}+\eps_{3}+\eps_{4}).$$
The lattices of integral weights of $G$ and $K$ are the same and equal to $P=\bbZ^{4}\cup((1/2,1/2,1/2,1/2)+\bbZ^{4})$ and the sets of dominant integral weights are denoted by $P^{+}_{G}$ and $P^{+}_{K}$.

\begin{theorem}\label{thm: multiplicity free condition}
There are three faces $F$ of $K$ such that $m^{G,K}_{\lambda}(\mu)\le1$ for all $\lambda\in P^{+}_{G}$ and all $\mu\in F$: the two dimensional face spanned by $\{\omega_{1},\omega_{2}\}$ and two one dimensional faces, spanned by $\omega_{3}$ and $\omega_{4}$ respectively.
\end{theorem}

This result has been obtained in \cite{He et al} as part of a classification. Another proof is given in \cite[Lem.~2.2.10]{van Pruijssen}.

The pair $(G,K)$ is a symmetric pair and choosing the maximal anisotropic torus $T\subset G$ (a circle group) as in \cite{Baldoni Silva} we have $Z_{K}(T)=M\cong\Spin(7)$, where the embedding $\Spin(7)\to\Spin(8)$ is twisted:
\begin{eqnarray}\label{embedding of M}
\laso(7,\bbC)\subset\laso(8,\bbC)\stackrel{\tau}{\to}\laso(8,\bbC)\subset\laso(9,\bbC),
\end{eqnarray}
with $\tau$ the automorphism that interchanges the roots $\eps_{1}-\eps_{2}$ and $\eps_{3}-\eps_{4}$, see \cite{Baldoni Silva}. We fix the maximal torus $\lah_{M}=\lam\cap\lah$ and choose the positive roots $\Delta(\lam,\lah_{M})$ such that the set of simple roots equals
$$\Pi_{M}=\{\delta_{1}=\eps_{3}-\eps_{4},\delta_{2}=\eps_{2}-\eps_{3},\delta_{3}=\frac{1}{2}(\eps_{1}-\eps_{2}+\eps_{3}+\eps_{4})\}.$$
The corresponding fundamental weights are given by
$$\eta_{1}=\frac{1}{2}(\eps_{1}+\eps_{2}+\eps_{3}-\eps_{4}),\eta_{2}=\eps_{1}+\eps_{2},\eta_{3}=\frac{1}{4}(3\eps_{1}+\eps_{2}+\eps_{3}+\eps_{4}).$$
The spherical weight is $\lambda_{\sph}=\varpi_{1}$. We want to calculate the map $P^{+}_{G}\to P^{+}_{M}$, but $\lambda_{\sph}$ is not perpendicular to $P^{+}_{M}$. Hence we pass to another Weyl chamber, and project along the new spherical direction, which is perpendicular to $P^{+}_{M}$. Choose a Weyl group element $w_{M}\in W_{G}$ such that the Weyl chamber $w_{M} P^{+}_{G}$ has the following properties: (1) $w_{M}\lambda_{\sph}\perp P_{M}$ and (2) the projection along $w_{M}\lambda_{\sph}$ induces a map $w_{M}P^{+}_{G}\to P^{+}_{M}$. We ask Mathematica \cite{Mathematica} to go through the list of Weyl group elements and test for these properties. We find two Weyl group elements, $w_{M}$ and $s_{1}w_{M}$, where

\begin{eqnarray}\label{eqn: element w}
w_M=\left(
\begin{array}{cccc}
 \frac{1}{2} & \frac{1}{2} & \frac{1}{2} & \frac{1}{2} \\
 -\frac{1}{2} & \frac{1}{2} & \frac{1}{2} & -\frac{1}{2} \\
 -\frac{1}{2} & \frac{1}{2} & -\frac{1}{2} & \frac{1}{2} \\
 -\frac{1}{2} & -\frac{1}{2} & \frac{1}{2} & \frac{1}{2}
\end{array}
\right)
\end{eqnarray}
with respect to the basis $\{\eps_{1},\eps_{2},\eps_{3},\eps_{4}\}$.

\begin{lemma}\label{lemma: M projection}
Let $q:P^{+}_{G}\to P^{+}_{M}$ be given by $q(\lambda)=w_{M}(\lambda)|_{\lah_{M}}$, where $w_{M}$ is given by (\ref{eqn: element w}). Then $q(P^{+}_{G}(\mu))=P^{+}_{M}(\mu)$ and  $q(\sum_{i=1}^{4}\lambda_{i}\varpi_{i})=\lambda_{4}\eta_{1}+\lambda_{3}\eta_{2}+\lambda_{2}\eta_{3}$.
\end{lemma}

\begin{proof}
The surjectivity is implied by Proposition \ref{prop: asymptotics}. The calculation involves a base change for $w_{M}$ with basis $\{\eta_{1},\eta_{2},\eta_{3},\alpha_{1}\}$ and follows readily.
\end{proof}

It follows that $\lambda=\lambda_{1}\varpi_{1}+\lambda_{2}\varpi_{2}+\lambda_{3}\varpi_{3}+\lambda_{4}\varpi_{4}\in P^{+}_{G}(\mu)$ implies that $\lambda_{4}\eta_{1}+\lambda_{3}\eta_{2}+\lambda_{2}\eta_{3}\in P^{+}_{M}(\mu)$. The branching rule $\Spin(9)\to\Spin(7)$ is described in \cite[Thm.~6.3]{Baldoni Silva} and we recall it for our special choices of $\mu$. It is basically the same as branching $B_{4}\downarrow D_{4}\downarrow B_{3}$ via interlacing, see e.g.~\cite[Thm.~9.16]{Knapp}, but on the $D_{4}$-level we have to interchange the coefficients of the first and the third fundamental weight.

\begin{proposition}\label{prop: branching spin9 -> spin7}
The spectrum $P^{+}_{M}(\mu)$ is given by the following inequalities.
\begin{itemize}\item Let $\mu=\mu_{1}\omega_{1}+\mu_{2}\omega_{2}$. Then $\lambda_{4}\eta_{1}+\lambda_{3}\eta_{2}+\lambda_{2}\eta_{3}\in P^{+}_{M}(\mu)$ if and only if
\begin{eqnarray*}
\lambda_{2}+\lambda_{3}+\lambda_{4}\le\mu_{1}+\mu_{2},\\
\lambda_{3}+\lambda_{4}\le\mu_{2}\le\lambda_{2}+\lambda_{3}+\lambda_{4}.
\end{eqnarray*}
\item Let $\mu=\mu_{2}\omega_{3}$. Then $\lambda_{4}\eta_{1}+\lambda_{3}\eta_{2}+\lambda_{2}\eta_{3}\in P^{+}_{M}(\mu)$ if and only if
\begin{eqnarray*}
\lambda_{2}+\lambda_{3}\le\mu_{3},\\
\lambda_{3}+\lambda_{4}\le\mu_{3}\le\lambda_{2}+\lambda_{3}+\lambda_{4}.
\end{eqnarray*}
\item Let $\mu=\mu_{4}\omega_{4}$. Then $\lambda_{4}\eta_{1}+\lambda_{3}\eta_{2}+\lambda_{2}\eta_{3}\in P^{+}_{M}(\mu)$ if and only if
\begin{eqnarray*}
\lambda_{3}=0,\\
\lambda_{2}+\lambda_{4}\le\mu_{4}.
\end{eqnarray*}

\end{itemize}
\end{proposition}

Given an element $\mu\in P^{+}_{K}$ we can determine the $M$-types $\nu=\nu_{1}\eta_{1}+\nu_{2}\eta_{2}+\nu_{3}\eta_{3}\in P^{+}_{M}(\mu)$ and we know from Proposition \ref{prop: asymptotics} that for $\lambda_{1}$ large enough,
\begin{eqnarray}\label{eqn: well limit}
\lambda=\lambda_{1}\varpi_{1}+\nu_{3}\varpi_{2}+\nu_{2}\varpi_{3}+\nu_{1}\varpi_{4}\in P^{+}_{G}(\mu).
\end{eqnarray}
We proceed to determine the minimal $\lambda_{1}$ such that (\ref{eqn: well limit}) holds, in the case that $\mu$ satisfies the multiplicity free condition of Theorem \ref{thm: multiplicity free condition}.

\begin{theorem}\label{thm: F4bottom}
Let $\mu\in(\bbN\omega_{1}\oplus\bbN\omega_{2})\cup(\bbN\omega_{3})\cup(\bbN\omega_{4})$. Then $\lambda=\lambda_{1}\varpi_{1}+\lambda_{2}\varpi_{2}+\lambda_{3}\varpi_{3}+\lambda_{4}\varpi_{4}\in B(\mu)$ if and only if (i) $q(\lambda)\in P^{+}_{M}(\mu)$ and (ii)
\begin{eqnarray}
\mu_{1}+\mu_{2}=\lambda_{1}+\lambda_{2}+\lambda_{3}+\lambda_{4} & \mbox{if} & \mu=\mu_{1}\omega_{1}+\mu_{2}\omega_{2},\label{eqn: bottom 1}\\
\mu_{3}=\lambda_{1}+\lambda_{2}+\lambda_{3} & \mbox{if} & \mu=\mu_{3}\omega_{3},\label{eqn: bottom 2}\\
\mu_{4}=\lambda_{1}+\lambda_{2}+\lambda_{4} & \mbox{if} & \mu=\mu_{4}\omega_{4}\label{eqn: bottom 3}.
\end{eqnarray}
\end{theorem}

Hence the bottom $B(\mu)$ is given by a singular equation and the inequalities of $P^{+}_{M}(\mu)$ in all cases, except for $(G,K)=(\SU(n+1),\mathrm{S}(\U(n)\times\U(1)))$. We have found the inequalities of Theorem \ref{thm: F4bottom} using an implementation of the branching rule from $\F_{4}$ to $\Spin(9)$ in Mathematica and looking at some examples. Before we prove Theorem \ref{thm: F4bottom} we settle the proof of the final case of Theorem \ref{degree inequality theorem}. 

\begin{corollary}
Let $\lambda\in P^{+}_{G}(\mu)\to\bbN$ and let $\lambda'\in P$ be a weight of the spherical representation. Then $|d(\lambda+\lambda')-d(\lambda|\le1$ with $d:P^{+}_{G}(\mu)\to\bbN$ the degree function of Theorem \ref{degree inequality theorem}. 
\end{corollary}

\begin{proof}
The weights of the spherical representation are the short roots and zero (with multiplicity two). After expressing these weights as linear combinations of fundamental weights, one easily checks the assertion.
\end{proof}

\begin{proof}[Proof of Theorem \ref{thm: F4bottom}.]
The proof of Theorem \ref{thm: F4bottom} is devided into two parts, corresponding to the dimension of the face. The strategy in both cases is the same.
Fix $\mu\in(\bbN\omega_{1}\oplus\bbN\omega_{2})\cup(\bbN\omega_{3})\cup(\bbN\omega_{4})$ and choose a suitable system $R^{+}_{G}$ of positive roots of $G$. Let $A=R^{+}_{G}\backslash R^{+}_{K}$ and let $p_{A}$ denote the corresponding partition function. Let $\lambda\in P^{+}_{G}$ have the property that $q(\lambda)\in P^{+}_{M}(\mu)$. This gives restrictions on $\lambda_{2},\lambda_{3},\lambda_{4}$, according to Proposition \ref{prop: branching spin9 -> spin7}. Let $\lambda_{1}$ satisfy the appropriate linear equation from the theorem.

For $w\in W_G$ define $\Lambda_{w}(\lambda,\mu)=w(\lambda+\rho)-(\mu+\rho)$. Explicit knowledge of the partition function $p_{A}$ allows us, using Mathematica, to determine for which $w\in W_G$ the quantity $p_{A}(\Lambda_{w}(\lambda,\mu))$ is zero. We end up with two elements in case $\mu\in\bbN\omega_{1}\oplus\bbN\omega_{2}$ and twelve elements in the other cases, for which $p_{A}(\Lambda_{w}(\lambda,\mu))$ is \textit{possibly} not zero. This allows us to calculate $m^{G,K}_{\lambda}(\mu)$ using Lemma \ref{branching rule lemma}.  One checks that the multiplicity is one for this choice of $\lambda\in P^{+}_{G}(\mu)$.

Moreover, if $\mu\in\bbN\omega_{1}\oplus\bbN\omega_{2}$ then $p(\Lambda_{w}(\lambda-\lambda_{\sph}))=0$ for all Weyl group elements. In the other cases for $\mu$ we find the same twelve Weyl group elements for which $p_{A}(\Lambda_{w}(\lambda-\lambda_{\sph},\mu))$ \textit{possibly} does not vanish. One checks that the multiplicity is zero in this case.

We conclude the proof by indicating the the positive system that we chose in the various cases, a description of the partition function and lists of the Weyl group elements that may contribute in the Kostant multiplicity formula.

\paragraph{The case $\mu=\mu_{1}\omega_{1}+\mu_{2}\omega_{2}$.} Here we take the standard positive system $R^{+}_{G}$ and we have $A=R_{G}^{+}\backslash R^{+}_{K}=\{\frac{1}{2}(\eps_{1}\pm\eps_{2}\pm\eps_{3}\pm\eps_{4})\}$. Let $\Lambda=(\Lambda_{1},\Lambda_{2},\Lambda_{3},\Lambda_{4})\in P$. We claim that $p_{A}(\Lambda)>0$ if and only if $|\Lambda_{j}|\le\Lambda_{1}$ for $j=2,3,4$.

Let us denote $A=\{a_{000},\ldots,a_{111}\}$ where the binary index indicates where to put the $+$ or the $-$ sign on positions 2,3,4, e.g. $a_{100}=\frac{1}{2}(\eps_{1}-\eps_{2}+\eps_{3}+\eps_{4})$. Let
\begin{eqnarray}\label{equality 12}
\Lambda=\sum_{i=000}^{111}n_{i}a_{i}.
\end{eqnarray}
We are going to count the number of tuples $(n_{000},\ldots,n_{111})\in\bbN^{8}$ for which (\ref{equality 12}) holds. First of all, it follows from (\ref{equality 12}) that
$$\sum_{i=000}^{011}n_{i}=\Lambda_{1}+\Lambda_{2},\quad\sum_{i=100}^{111}n_{i}=\Lambda_{1}-\Lambda_{2}.$$
In other words, any linear combination (\ref{equality 12}) uses $\Lambda_{1}+\Lambda_{2}$ elements from the set $\{a_{000},\ldots,a_{011}\}$ and $\Lambda_{1}-\Lambda_{2}$ elements from the set $\{a_{100},\ldots,a_{111}\}$. Let us write $(\Lambda_{3},\Lambda_{4})=(v_{1},v_{2})+(\Lambda_{3}-v_{1},\Lambda_{4}-v_{2})$. For each such decomposition we need to count (1) the number of tuples $(n_{000},\ldots,n_{011})\in\bbN^{4}$ for which $\sum_{i=000}^{011}n_{i}a_{i}=((\Lambda_{1}+\Lambda_{2})/2,(\Lambda_{1}+\Lambda_{2})/2,v_{1},v_{2})$ and (2) the number of tuples $(n_{100},\ldots,n_{111})\in\bbN^{4}$ for which $\sum_{i=100}^{111}n_{i}a_{i}=((\Lambda_{1}-\Lambda_{2})/2,-(\Lambda_{1}-\Lambda_{2})/2,\Lambda_{3}-v_{1},\Lambda_{4}-v_{2})$. For each $(v_{1},v_{2})$ we take the product of these quantities, and summing these for the possible vectors $(v_{1},v_{2})$ yields the desired formula for $p_{A}$.

This reduces the calculation of $p_{A}$ to the following counting problem. Let $L=\bbZ^{2}\cup((\frac{1}{2},\frac{1}{2})+\bbZ^{2})$, let $A'=\{(\pm\frac{1}{2},\pm\frac{1}{2})\}$ and let $p\in\bbN$. Let us denote $A'=\{a'_{00},\ldots,a'_{11}\}$, where the binary number indicates where to put the $+$ and the $-$ signs, e.g. $a'_{10}=(-1/2,1/2)$. Given a vector $v=(v_{1},v_{2})$ we want to calculate the number of tuples $(n_{00},\ldots,n_{11})\in\bbN^{4}$ such that $\sum_{i=00}^{11}n_{i}a'_{i}=v$ and $\sum_{i=00}^{11}n_{i}=p$. It is necessary that $|v_{1}|,|v_{2}|\le p/2$. In this case, the number of tuples is $1+\frac{p}{2}-\max(|v_{1}|,|v_{2}|)$.

Returning to our original problem, we have
\begin{multline*}
p_{A}(\Lambda)=\sum_{v_{1},v_{2}}\left(1+\frac{\Lambda_{1}+\Lambda_{2}}{2}-\max(|v_{1}|,|v_{2}|)\right)\times \\ \left(1+\frac{\Lambda_{1}-\Lambda_{2}}{2}-\max(|\Lambda_{3}-v_{1}|,|\Lambda_{4}-v_{2}|)\right),
\end{multline*}
where $(v_{1},v_{2})$ satisfies the restrictions $|v_{1}|,|v_{2}|\le (\Lambda_{1}+\Lambda_{2})/2$ and simultaneously $|\Lambda_{3}-v_{1}|,|\Lambda_{4}-v_{2}|\le (\Lambda_{1}-\Lambda_{2})/2$. As a result, the ranges for the summations are
$$
v_{1}=\max\left(-\frac{\Lambda_{1}+\Lambda_{2}}{2},\Lambda_{3}-\frac{\Lambda_{1}-\Lambda_{2}}{2}\right),\ldots,\min\left(\frac{\Lambda_{1}+\Lambda_{2}}{2},\Lambda_{3}+\frac{\Lambda_{1}-\Lambda_{2}}{2}\right),$$
$$
v_{2}=\max\left(-\frac{\Lambda_{1}+\Lambda_{2}}{2},\Lambda_{4}-\frac{\Lambda_{1}-\Lambda_{2}}{2}\right),\ldots,\min\left(\frac{\Lambda_{1}+\Lambda_{2}}{2},\Lambda_{4}+\frac{\Lambda_{1}-\Lambda_{2}}{2}\right).
$$

In particular, $p_{A}(\Lambda)>0$ if and only if the ranges for $v_{1}$ and $v_{2}$ are both non-empty, which is equivalent to
\begin{eqnarray}
|\Lambda_{2}| & \le & \Lambda_{1}\label{ineq: A standard 1},\\
|\Lambda_{3}| & \le & \Lambda_{1}\label{ineq: A standard 2},\\ 
|\Lambda_{4}| & \le & \Lambda_{1}\label{ineq: A standard 3}.
\end{eqnarray}
The only two Weyl group elements for which $p_{A}(\Lambda_{w}(\lambda,\mu))$ contributes to the multiplicity $m^{G,K}_{\lambda}(\mu)$, under the assumptions (\ref{eqn: bottom 1}) and $q(\lambda)\in P^{+}_{M}(\mu)$ are $e,s_{2}$. In this case $m^{G,K}_{\lambda}(\mu)=1$. Also, $m^{G,K}_{\lambda-\lambda_{\sph}}(\mu)=0$ under the same conditions, as there are no Weyl group elements for which $p_{A}(\Lambda_{w}(\lambda-\lambda_{\sph},\mu))$ is non-zero.

\paragraph{The case $\mu=\mu_{3}\omega_{3}$ and $\mu=\mu_{4}\omega_{4}$.}

Let $\mu=\mu_{3}\omega_{3}$ or $\mu=\mu_{4}\omega_{4}$ and $\lambda\in B(\mu)$ and consider $\Lambda_{w}(\lambda,\mu)=w(\lambda+\rho)-(\mu+\rho)$ for $w\in W_{G}$. Using Mathematica to check the inequalities (\ref{ineq: A standard 1},\ref{ineq: A standard 2},\ref{ineq: A standard 3}) under the condition (\ref{eqn: bottom 2}) or (\ref{eqn: bottom 3}) we find a number of 16 Weyl group elements for which $\Lambda_{w}(\lambda,\mu)$ is possibly in the support of $p_{A}$. However, the formulas for the elements $\Lambda_{w}(\lambda,\mu)$ that possibly contribute do not look tempting to perform calculations with.

Instead we pass to another Weyl chamber for $\F_{4}$ while remaining in the same Weyl chamber for $\Spin(9)$. The Weyl chamber that we choose contains $\omega_{3}$ and $\omega_{4}$. The element $\widetilde{w}=s_{2}s_{1}\in W$ translates the standard Weyl chamber to one that we are looking for. The set of positive roots that corresponds to the system of simple roots is $\widetilde{w}\Pi_{G}=R^{+}_{K}\cup B$, where
\begin{multline*}B=\left\{
\frac{1}{2}(-\eps_{1}+\eps_{2}+\eps_{3}\pm\eps_{4}),
\frac{1}{2}(\eps_{1}-\eps_{2}+\eps_{3}\pm\eps_{4}),\right.\\
\left.\frac{1}{2}(\eps_{1}+\eps_{2}-\eps_{3}\pm\eps_{4}),
\frac{1}{2}(\eps_{1}+\eps_{2}+\eps_{3}\pm\eps_{4})\right\}
\end{multline*}
is the new set of positive roots of $G$ that are not roots of $K$. The Kostant multiplicity formula reads
$$m^{G,K}_{\lambda}(\mu)=\sum_{w\in W_{G}}\det(w)p_{B}(\Lambda_{w}(w(\lambda+\widetilde{\rho})-(\mu+\widetilde{\rho})),$$
where $\widetilde{\rho}=\frac{1}{2}(9\eps_{1}+7\eps_{2}+5\eps_{3}+\eps_{4})$ is the Weyl vector for the new system of positive roots. 

Our aim is to calculate the partition $p_{B}(\Lambda)$ for $\Lambda=(\Lambda_{1},\Lambda_{2},\Lambda_{3},\Lambda_{4})\in P$. To begin with we focus on the first three coordinates. Let $\pi:P\to\bbZ^{3}\cup((\frac{1}{2},\frac{1}{2},\frac{1}{2})+\bbZ^{3})$ denote the projection on the first three coordinates. Let $C=\{c_{1},c_{2},c_{3},c_{4}\}$ with
\begin{multline*}
c_{1}=\frac{1}{2}(-\eps_{1}+\eps_{2}+\eps_{3}),
c_{2}=\frac{1}{2}(\eps_{1}-\eps_{2}+\eps_{3}),\\
c_{3}=\frac{1}{2}(\eps_{1}+\eps_{2}-\eps_{3}),
c_{4}=\frac{1}{2}(\eps_{1}+\eps_{2}+\eps_{3}).$$
\end{multline*}
The number of linear combinations $\pi(\Lambda)=n_{1}c_{1}+n_{2}c_{2}+n_{3}c_{3}+n_{4}c_{4}$ with $n_{i}\in\bbN$ is non-zero if and only if
\begin{eqnarray}
0 & \le & \Lambda_{1}+\Lambda_{2},\label{ineq: B 1}\\
0 & \le & \Lambda_{1}+\Lambda_{3},\label{ineq: B 2}\\
0 & \le & \Lambda_{2}+\Lambda_{3}.\label{ineq: B 3}
\end{eqnarray}
We assume $\Lambda_{1}\ge\Lambda_{2}\ge\Lambda_{3}$. We have
\begin{eqnarray*}
(\Lambda_{1},\Lambda_{2},\Lambda_{3})&=&(\Lambda_{1}-\Lambda_{2})c_{1}+(\Lambda_{1}-\Lambda_{3})c_{2}+(\Lambda_{2}+\Lambda_{3})c_{4}\\
&=&(\Lambda_{1}-\Lambda_{2}+1)c_{1}+(\Lambda_{1}-\Lambda_{3}+1)c_{2}+c_{3}+(\Lambda_{2}+\Lambda_{3}-1)c_{4}\\
&\vdots&\\
&=&(\Lambda_{1}+\Lambda_{3})c_{1}+(\Lambda_{1}+\Lambda_{2})c_{2}+(\Lambda_{2}+\Lambda_{3})c_{3},
\end{eqnarray*}
from which we see that there are $\Lambda_{2}+\Lambda_{3}+1$ ways to write $(\Lambda_{1},\Lambda_{2},\Lambda_{3})$ as a linear combination of elements in $C$ with coefficients in $\bbN$. Every such combination uses a unique number of vectors: $2\Lambda_{1}+2r$, where $r=0,\ldots,\Lambda_{2}+\Lambda_{3}$.

Let $b_{i,\pm}=c_{i}\pm\frac{1}{2}\eps_{4}$ denote the elements in $B$ that project onto $c_{i}\in C$. Let $\Lambda=\sum s_{i,\pm}b_{i,\pm}$ be a positive integral linear combination of elements in $B$ and define $m_{i}=s_{i,+}+s_{i,-}$. Then $\pi(\Lambda)=\sum m_{i}c_{i}$ is a linear combination of elements in $C$ with coefficients in $\bbN$ and hence there is an $r\in\{0,\ldots,\Lambda_{2}+\Lambda_{3}\}$ such that $m_{1}=\Lambda_{1}-\Lambda_{3}+r,m_{2}=\Lambda_{1}-\Lambda_{2}+r,m_{3}=r$ and $m_{4}=\Lambda_{2}+\Lambda_{3}-r$. We find that $\sum_{i=1}^{4}s_{i,+}-\sum_{i=1}^{4}s_{i,-}=2\Lambda_{4}$ and $\sum_{i=1}^{4}s_{i,+}+\sum_{i=1}^{4}s_{i,-}=2\Lambda_{1}+2r$. It follows that the number of ways in which we can write $\Lambda$ as a linear combination of $2\Lambda_{1}+2r$ elements in $B$ with coefficients in $\bbN$ is equal to the number of tuples $(s_{1,+},s_{2,+},s_{3,+},s_{4,+})\in\bbN^{4}$ with $\sum_{i=1}^{4}s_{i,+}=\Lambda_{1}+\Lambda_{4}+r$ and $0\le s_{i,+}\le m_{i}$. This is the number of integral points in the intersection of the hyperrectangular $\{0\le s_{i,+}\le m_{i}\}$ and the affine hyperplane $\{s_{1,+}+s_{2,+}+s_{3,+}+s_{4,+}=\Lambda_{1}+\Lambda_{4}+r\}$ and we denote this quantity with $L((m_{1},m_{2},m_{3},m_{4}),\Lambda_{1}+\Lambda_{4}+r)$. Whenever
\begin{eqnarray}\label{ineq: B 4}
|\Lambda_{4}|\le \Lambda_{1}+\Lambda_{2}+\Lambda_{3},
\end{eqnarray}
$L((m_{1},m_{2},m_{3},m_{4}),\Lambda_{1}+\Lambda_{4}+r)>0$. Hence
$$p_{B}(\Lambda)=\sum_{r=0}^{\Lambda_{2}+\Lambda_{3}}L((\Lambda_{1}-\Lambda_{3}+r,\Lambda_{1}-\Lambda_{2}+r,r,\Lambda_{2}+\Lambda_{3}-r),\Lambda_{1}+\Lambda_{4}+r)$$
if $\Lambda_{1}\ge\Lambda_{2}\ge\Lambda_{3}$. The quantity $p_{B}(\Lambda)$ is positive if and only if the inequalities (\ref{ineq: B 1}),(\ref{ineq: B 2}),(\ref{ineq: B 3}) and (\ref{ineq: B 4}) hold. Note that these inequalities are invariant for permuting the first three coordinates of $\Lambda$.

Let $\mu=\mu_{3}\omega_{3}$ or $\mu=\mu_{4}\omega_{4}$ and let $\lambda\in P^{+}_{G}$ satisfy $q(\lambda)\in P^{+}_{M}(\mu)$ and (\ref{eqn: bottom 2}) or (\ref{eqn: bottom 3}) respectively. Define $\Gamma_{w}(\lambda,\mu)=w(\widetilde{w}\lambda+\widetilde{\rho})-(\mu+\widetilde{\rho})$. For the elements $\Gamma_{w}(\lambda,\mu)$ and $\Gamma_{w}(\lambda-\lambda_{\sph},\mu)$ we check the inequalities (\ref{ineq: B 1}),(\ref{ineq: B 2}),(\ref{ineq: B 3}) and (\ref{ineq: B 4}). We get 12 Weyl group elements for which $p_{B}(\Gamma_{w}(\lambda,\mu))$ and $p_{B}(\Gamma_{w}(\lambda-\lambda_{\sph},\mu))$ are possibly non-zero. Moreover, the twelve elements are the same for $\mu=\mu_{3}\omega_{3}$ and $\mu=\mu_{4}\omega_{4}$ and we have listed them in Table \ref{table: Weyl group elements}.

\begin{table}[ht]
\begin{center}
\begin{tabular}{|c|c|c|c|c|c|c|c|c|c|c|c|} \hline
$ w_{1}  $ & $  w_{2}  $ & $  w_{3}  $ & $  w_{4}  $ & $  w_{5}  $ & $  w_{6}  $ & $  w_{7}  $ & $  w_{8}  $ & $  w_{9}  $ & $  w_{10}  $ & $  w_{11}  $ & $  w_{12}$  \\ \hline
$ e          $ & $  s_{1}   $ & $  s_{2}   $ & $  s_{3}   $ & $  s_{4}   $ & $ s_{1}s_{2}  $ & $  s_{1}s_{3}  $ & $  s_{1}s_{4}  $ & $  s_{2}s_{1}  $ & $  s_{2}s_{3}  $ & $  s_{2}s_{4}  $ & $  s_{3} s_{2}$ \\ \hline
\end{tabular}
\caption{The Weyl group elements $w$ for which $\Gamma_{w}(\widetilde{w}\lambda,\mu)$ is possibly in the support of $p_{B}$.}\label{table: Weyl group elements}
\end{center}
\end{table}
\noindent Using the explicit description of $p_{B}$ one verifies the equalities $m^{G,K}_{\lambda}(\mu)=1$ and $m^{G,K}_{\lambda-\lambda_{\sph}}(\mu)=0$.
\end{proof}


\section{The differential equations}\label{DEs}

Our goal is to define a non-trivial commutative algebra of differential operators for the matrix valued orthogonal polynomials defined in Section \ref{intro}. Let $(G,K,F)$ be a multiplicity free system from Table \ref{table: mfs} and let $\mu\in F$. Let $\lag_{c},\lak_{c}$ denote the complexifications, let $U(\lag_{c})$ denote the universal enveloping algebra of $\lag_{c}$ and let $U(\lag_{c})^{\lak_{c}}$ denote the commutant of $\lak_{c}$ in $U(\lag_{c})$.
Let $\pi_{\mu}^{K}$ be an irreducible representation of $K$ in $V_{\mu}$ and let $\dot{\pi}^{K}_{\mu}$ denote the corresponding representation of $U(\lak_{c})$. Let $I(\mu)\subset U(\lak_{c})$ denote the kernel of $\dot{\pi}^{K}_{\mu}$ and consider the left ideal $U(\lag_{c})I(\mu)\subset U(\lag_{c})$. As in \cite[Ch.~9]{Dixmier} we define
$$\bbD(\mu)=U(\lag_{c})^{\lak_{c}}/(U(\lag_{c})^{\lak_{c}}\cap U(\lag_{c})I(\mu)),$$
which is an associative algebra. In fact, $\bbD(\mu)$ is commutative because it can be embedded, using an anti homomorphism, into the commutative algebra $U(\laa_{c})\otimes\End_{M}(V_{\mu})$ (see \cite[9.2.10]{Dixmier}), which is commutative by Proposition \ref{prop: reducing to M}. The irreducible representations of $\bbD(\mu)$ are in a 1--1 correspondence with the irreducible representations of $\lag_{c}$ that contain $\dot{\pi}^{K}_{\mu}$ upon restriction, see \cite[Thm.~9.2.12]{Dixmier}.

Let $D\in U(\lag_{c})$. The $\mu$-radial part $R(\mu,D)$ is a differential operator that satisfies
\begin{eqnarray}\label{def:radial part}
R(\mu,D)(\Phi|_{T})=D(\Phi)|_{T}
\end{eqnarray}
for all functions $\Phi:G\to\End(V_{\mu})$ satisfying (\ref{trafo rule}).
Following Casselman and Mili{\v{c}}i{\'c} \cite[Thm.~3.1]{CM1982} we find a homomorphism 
$$R(\mu): U(\lag_{c})^{\lak_{c}}\to C(T)\otimes U(\lat_{c})\otimes\End(\End_{M}(V_{\mu}))$$
such that (\ref{def:radial part}) holds for all $D\in U(\lag_{c})^{\lak_{c}}$ and all $\Phi\in C^{\infty}(G,\End(V_{\mu}))$ satisfying (\ref{trafo rule}). For the two non-symmetric multiplicity free triples we have an Iwasawa-like decomposition $\lag_{c}=\lak_{c}\oplus\lat_{c}\oplus\lan^{+}$ and a map $\lan^{+}\to\lak_{c}$ onto the orthocomplement of $\lam_{c}$ in $\lak_{c}$. This map replaces $I+\theta$ in the symmetric case and is essential in the construction of $R_{\mu}$, see \cite[Lem.~2.2]{CM1982}. The homomorphism $R(\mu)$ factors through the projection $U(\lag_{c})^{\lak_{c}}\to\bbD(\mu)$ and we obtain an injective algebra homomorphism that we denote by the same symbol, $R(\mu):\bbD^{\mu}\to C(T)\otimes U(\lat_{c})\otimes\End(\End_{M}(V_{\mu}))$. We identify $\End_{M}(V_{\mu})=\bbC^{N_{\mu}}$ by Schur's Lemma with $N_{\mu}$ the cardinality of the bottom $B(\mu)$ and we write $\mathbb{M}^{\mu}=\End(\bbC^{N_{\mu}})$. The elementary spherical functions $\Phi^{\mu}_{\lambda}$ are simultaneous eigenfunctions for the algebra $\bbD(\mu)$. The differential operators $R(\mu,D)$ become differential operators for the functions $\Psi_{d}^{\mu}:T\to\mathbb{M}^{\mu}$ and, according to the construction, the functions $\Psi_{n}^{\mu}$ are simultaneous eigenfunctions for the operators $R(\mu,D)$ with $D\in\bbD(\mu)$. The eigenvalues are diagonal matrices $\Lambda_{n}(D)\in\mathbb{M}^{\mu}$ acting on the right, i.e.~we have $R(\mu,D)\Psi^{\mu}_{n}=\Psi^{\mu}_{n}\Lambda_{n}(D)$.

In the forthcoming paper \cite{van Pruijssen Roman} it is shown that the function $\Psi^{\mu}_{0}:T\to\mathbb{M}^{\mu}$ is point wise invertible on $T_{\reg}$, the open subset of $T$ on which the restriction of the minimal spherical function, $\phi|_{T}$, is regular. The proof relies on the bispectral property that is present for the family of matrix valued functions $\{\Psi^{\mu}_{n}:n\in\bbN\}$. More precisely, the interplay between the differential operators and the three term recurrence relation imply that the function $\Psi^{\mu}_{0}$ satisfies an ODE whose coefficients are regular on $T_{\reg}$. If we conjugate $R(\mu,D)$ with $\Psi^{\mu}_{0}$ and perform the change of variables $x=c\phi(t)+(1-c)$, such that $x$ runs in $[-1,1]$, then we obtain a differential operator acting on the space of matrix valued orthogonal polynomials $\mathbb{M}^{\mu}[x]$. The algebra of differential operators that is obtained in this way is denoted by $\bbD^{\mu}$. The family of matrix valued orthogonal polynomials $(P^{\mu}_{n}(x);n\in\bbN)$ that we obtain from the functions $(\Psi^{\mu}_{n};n\in\bbN)$, is a family of simultaneous eigenfunctions for the algebra $\bbD^{\mu}$. The algebra of differential operators $\mathbb{M}^{\mu}[x,\del_{x}]$ acts on $\mathbb{M}^{\mu}[x]$, where the matrices act by left multiplication. Note that $\bbD^{\mu}\subset\mathbb{M}^{\mu}[x,\del_{x}]$. 

The description of the map $R(\mu)$ in \cite{CM1982} allows one to calculate explicitly the radial part of the (order two) Casimir operator $\Omega\in U(\lag_{c})^{\lak_{c}}$. An explicit expression can be found in \cite[Prop.~9.1.2.11]{Warner} for the case where $(G,K)$ is symmetric. The image of $\Omega$ in the algebra $\bbD^{\mu}$ is denoted by $\Omega^{\mu}$ and is of order two. Its eigenvalues can be calculated explicitly in terms of highest weights and they are real, which implies that $\Omega^{\mu}$ is symmetric with respect to the matrix valued inner product $\langle\cdot,\cdot\rangle_{W^{\mu}}$. These are examples of matrix valued hypergeometric differential operators \cite{Tirao}.


\section{Conclusions}

Several questions remain. We have shown the existence of families of matrix valued orthogonal polynomials, together with a commutative algebra of differential operators for which the polynomials are simultaneous eigenfunctions, mainly by working out the branching rules. The key result is that the bottom of the $\mu$-well is well behaved with respect to the weights of the fundamental spherical representation, so that the degree function has the right properties. It would be interesting to see whether one can draw the same conclusions by investigating of the differential equations for the matrix valued orthogonal polynomials. This would require more precise knowledge of the algebra $\bbD(\mu)$.

On the other hand, it would be interesting to investigate whether the good properties of the degree function follow from convexity arguments that come about if we formulate matters concerning the representation theory, such as induction and restriction, in terms of symplectic or algebraic geometry. For example, in this light, it is interesting to learn more about the (spherical) spaces $G_{c}/Q$ and their $G_{c}$-equivariant line bundles, where $Q\subset K_{c}$ is the parabolic subgroup associated to $F$, for a multiplicity free system $(G,K,F)$.

The existence of multiplicity free systems $(G,K,F)$ with $(G,K)$ a Gelfand pair of rank $>1$, raises the question whether the spectra of the induced representations have a similar structure as in the rank one case. If the answer is affirmative we expect that we can associate families of matrix valued orthogonal polynomials in several variables to these spectra, together with commutative algebras of differential operators that have these polynomials as simultaneous eigenfunctions. For the examples $(\Spin(9),\Spin(7),\bbN_{\omega_{1}})$ and $(\SU(n+1)\times\SU(n+1),\diag(\SU(n+1)),F)$, where $F=\omega_{1}\bbN$ or $F=\omega_{n}\bbN$, this seems to be the case. In general the branching rules will not be of great help in understanding the bottom of the $\mu$-well, as they soon become too complicated in the higher rank situations.


Gert Heckman, Radboud Universiteit Nijmegen, IMAPP, P.O.~Box 9010, 6500 GL Nijmegen, The Netherlands (\verb|g.heckman@math.ru.nl|).

Maarten van Pruijssen, Universit{\"a}t Paderborn, Institut f{\"u}r Mathematik, Warburger Str. 100, 33098 Paderborn, Germany (\verb|vanpruijssen@math.upb.de|).

\end{document}